\documentclass[12pt]{article}

\usepackage{amssymb,amsmath,amsfonts,amssymb,bbm}
\usepackage{graphics,graphicx,xcolor}
\usepackage{empheq}

\usepackage[colorlinks=true, urlcolor=violet, linkcolor=blue, citecolor=red, hyperindex=true, linktocpage=true]{hyperref}

\def\@abssec#1{\vspace{.05in}\footnotesize \parindent .2in
{\bf #1. }\ignorespaces}

\graphicspath{{/EPSF/}{Figures/}}
\DeclareGraphicsExtensions{.eps}

\numberwithin{equation}{section}

\newtheorem{theorem}{Theorem}[section]
\newtheorem{lemma}[theorem]{Lemma}
\newtheorem{proposition}[theorem]{Proposition}
\newtheorem{corollary}[theorem]{Corollary}
\newtheorem{definition}[theorem]{Definition}

\def \Rm {\mathbb R}
\def \Pm {\mathbb P}
\def \Nm {\mathbb N}

\def \Zm {\mathbb Z}
\def \Sm {\mathbb S}

\newcommand{\eps}{\varepsilon}

\newcommand{\dsum}{\displaystyle\sum}
\newcommand{\dint}{\displaystyle\int}

\newcommand{\aver}[1]{\langle {#1} \rangle}

\newcommand{\BC}{{\rm b}}
\newcommand{\CB}{{\rm d}}

\newcommand{\mA}{\mathcal A}

\newcommand{\mF}{\mathcal F}
\newcommand{\mG}{\mathcal G}
\newcommand{\mH}{\mathcal H}

\newcommand{\mD}{\mathfrak D}

\newcommand{\fa}{{\mathfrak a}}
\newcommand{\fb}{{\mathfrak b}}
\newcommand{\fc}{{\mathfrak c}}

\newcommand{\fO}{{\mathfrak O}}

\newcommand{\cout}[1]{}

\newcommand{\one}{{\mathbbm{1}}}

 \renewcommand{\arraystretch}{1.5}
\newcommand{\ml}[1]{{\color{black} #1}}

\title{\textbf{Homogenization of hydrodynamic transport in Dirac fluids}}

\author{Guillaume Bal\thanks{Departments of Statistics and Mathematics and CCAM, University of Chicago, Chicago, IL 60637; guillaumebal@uchicago.edu. Supported in part by the US NSF and ONR.},
Andrew Lucas\thanks{Department of Physics and Center for the Theory of Quantum Matter, University of Colorado, Boulder, CO 80309; andrew.j.lucas@colorado.edu},
Mitchell Luskin\thanks{School of Mathematics, University of Minnesota, Minneapolis, MN 55455; luskin@umn.edu. Supported in part by ARO MURI Award W911NF-14-0247, NSF DMREF Award 1922165, and NSF Award DMS-1819220.}}

\begin{document}

\maketitle


\begin{abstract}


Large-scale electrical and thermal currents in ordinary metals are well approximated by effective medium theory: global transport properties are governed by the solution to homogenized coupled diffusion equations.  In some metals, including the Dirac fluid of nearly charge neutral graphene, microscopic transport is not governed by diffusion, but by a more complicated set of linearized hydrodynamic equations, which form a system of degenerate elliptic equations coupled with the Stokes equation for fluid velocity.  In sufficiently inhomogeneous media, these hydrodynamic equations reduce to homogenized diffusion equations.

We re-cast the hydrodynamic transport equations as the infimum of a functional over conserved currents, and present a functional framework to model and compute the homogenized diffusion tensor relating electrical and thermal currents to charge and temperature gradients.  We generalize to this system two well-known results in homogenization theory: Tartar's proof of local convergence to the homogenized theory in periodic and highly oscillatory media, and sub-additivity of the above functional in random media with highly oscillatory, stationary and ergodic coefficients.


\end{abstract}


\renewcommand{\thefootnote}{\fnsymbol{footnote}}
\renewcommand{\thefootnote}{\arabic{footnote}}

\renewcommand{\arraystretch}{1.1}





\tableofcontents

\section{Introduction}
\subsection{Hydrodynamic transport}
 In ordinary metals, the flow of electrical and thermal currents is described by a pair of coupled diffusion equations: \begin{equation}
\nabla \cdot \left(\left(\begin{array}{cc} \sigma &\ \alpha \\ \tilde\alpha &\ \tilde\kappa \end{array}\right) \nabla \left(\begin{array}{c} \mu \\ T \end{array}\right)\right) =0, \label{eq:introdiff}
\end{equation}
where $ \sigma, \alpha, \tilde\alpha, \tilde\kappa : \mathbb{R}^D \rightarrow \mathbb{R}^{D\times D}$ are thermoelectric conductivity matrices which may depend on position.  The entire $2D\times 2D$ matrix above must be positive definite.    This simple set of equations describes nearly all macroscopic thermoelectric transport phenomena, and has been known for over a century.

At microscopic scales, there is no reason for the coefficients such as $\sigma$ to be independent of position.   Metals are dirty and inhomogeneous.  A toy model for this inhomogeneity is to take the conductivity matrices such as $\sigma$ to depend on position.  Yet that is not generally correct in real materials:  the microscopic equations of motion are often substantially more complicated than a simple diffusion equation.   A simplifying regime occurs when electrons become strongly correlated and begin to flow hydrodynamically \cite{andreev, gurzhi}; see \cite{lucasreview} for a review.  As a simple example, the flow of electrical and thermal currents in nearly charge neutral graphene is described by (linearized) Dirac fluid hydrodynamics \cite{fritz, hartnoll2007, lucas2016}:  (\ref{eq:introdiff}) becomes replaced by \begin{equation}\label{eq:systWF}
\begin{array}{rcl}
  \nabla\cdot \Big( -\sigma_Q(\nabla\mu-\gamma\nabla T) +n v\Big) &=& 0, \\[2mm]
  \nabla\cdot \Big( \sigma_Q\gamma(\nabla\mu-\gamma\nabla T) + s v \Big)  &=& 0 ,\\[2mm]
  n \nabla \mu + s \nabla  T - \eta \Delta  v - (\nabla \zeta\nabla\cdot) v&=& 0,
\end{array}
\end{equation}
where $\mu$ is the (perturbed) chemical potential, $T$ is the (perturbed) temperature, and $v$ is the velocity field.
The material parameters are the quantum conductivity $\sigma_Q>0$, viscosity coefficients $\eta>0$ and $\zeta>0,$ spatially varying chemical potential  $\gamma(x)$, charge density $n=n(\gamma),$ and entropy density $s=s(\gamma).$\footnote{In principle, $\sigma_Q,\eta,\zeta$ may also depend on $\gamma$, but this is not necessary in order to have a well-posed problem, and also does not change the strategy of solution.}    

The first two equations of (\ref{eq:systWF}) represent the conservation of charge and heat, respectively.  We emphasize that both of these equations do not take the usual form \cite{landau}; this is due to the lack of Galilean symmetry of the microscopic dynamics.  Instead, there is an approximate Lorentz symmetry which fixes the form of these equations \cite{hartnoll2007}.  Thermodynamic identities $n=\partial P/\partial \mu$ and $s=\partial P/\partial T$ (here $P=P(\mu,T)$ is the pressure) imply that the third equation of (\ref{eq:systWF}) is simply the time-independent linearized Navier-Stokes equation since 
\begin{equation}
n\nabla\mu + s\nabla T = \nabla P. \label{eq:gradpressure}
\end{equation}
We now can provide more motivation for the specific form (\ref{eq:systWF}): the conductor has a locally inhomogeneous chemical potential (caused by charged impurities) -- this chemical potential modulation $\gamma$ is responsible for fluctuations in the thermodynamic coefficients $n$ and $s$, which are related by (\ref{eq:gradpressure}).

All three of these equations are to be understood as the leading order terms in a derivative expansion (called effective theory in physics) \cite{lucasreview}.  In general, hydrodynamics describes only the dynamics of the degrees of freedom whose wavelength $\lambda$ is very large compared to a microscopic mean free path $\ell$:  $\ell \ll \lambda$, and it should be thought of as an asymptotic expansion in the parameter $\ell/\lambda \rightarrow \ell \nabla$.   The terms which are included in (\ref{eq:systWF}) are the leading order terms (fewest number of derivatives) which lead to a well-posed problem.

The Dirac fluid model above predicts a number of unconventional transport phenomena, such as the breakdown of the Wiedemann-Franz law: a relationship between two of the matrices in (\ref{eq:introdiff}):\footnote{We note that the experimentally measured ``thermal conductivity" is generally not the coefficient $\tilde\kappa$: see (\ref{eq:truekappa}).  The Wiedemann-Franz law is more conventionally stated in terms of the experimentally measured coefficient, but (\ref{eq:wflaw}) also holds as stated in ordinary metals, at low temperature. } \begin{equation}
    \tilde\kappa = \frac{\pi^2 T_0}{3} \sigma.  \label{eq:wflaw}
\end{equation}
We have set some fundamental constants of nature to unity; here $T_0$ represents the equilibrium temperature of the sample, which has been scaled to $1$ in \eqref{eq:systWF} \cite{lucas2016}).  This relationship is expected whenever the dominant microscopic scattering mechanism is elastic scattering of single electrons (often off of impurities, at low temperatures).   There is no reason for this relation to hold in (\ref{eq:systWF}) -- and indeed, it does not -- (\ref{eq:systWF}) describes the collective hydrodynamic flow of an electron fluid through an inhomogeneous landscape.  Since experimentally, the violations of (\ref{eq:wflaw}) are a clear signature of unconventional transport, it is worthwhile to understand exactly how much $\kappa$ and $\sigma$ can differ.

\subsection{Homogenization}
A typical application of unconventional transport models such as (\ref{eq:systWF}), in the physics literature, proceeds as follows.  One first solves the system (\ref{eq:systWF}) in a periodic domain, and calculates \emph{global} transport properties.  After finding the global transport coefficients, one then \emph{asserts} that those coefficients (appropriately rescaled) are equivalent to \emph{local} transport coefficients in (\ref{eq:introdiff}), such as $\sigma$.  Finally, one can solve the transport problem in any domain $X$ with any boundary conditions, not by solving (\ref{eq:systWF}) but by solving the simpler system (\ref{eq:introdiff}). 

The purpose of this paper is to prove that these assumptions are justified by showing that the solution of (\ref{eq:systWF}), in a sufficiently large domain with $\gamma$ sufficiently oscillatory in all directions, \emph{locally converges} to a solution of (\ref{eq:introdiff}) with \emph{constant coefficients}, equal to those obtained through a global solution.  In simpler settings, including \eqref{eq:introdiff} along with more general systems, convergence proofs are achieved through a well understood mathematical framework called homogenization \cite{BLP-78,JKO-SV-94}.  We will extend these results to systems such as \eqref{eq:systWF}, and provide a mathematical perspective on earlier theories of hydrodynamic transport that have arisen in the physics literature \cite{andreev, lucas2015, lucas2016, lucashartnoll2017}.

The system (\ref{eq:systWF}) includes a number of obstructions to the standard homogenization theory.  For conventional systems including (\ref{eq:introdiff}), the equations themselves are the infimum of a Lagrangian functional and can be studied with standard minimization procedures.  In contrast, the Lagrangian which leads to (\ref{eq:systWF}) is not bounded from below:  the equations of motion are a \emph{saddle point}.  A more convenient variational formulation exists for the conserved quantities in the system, namely the (charge and thermal) currents.  In the physics literature, it was shown that this variational principle is interpreted as the statement that transport proceeds along the least dissipative trajectory, consistent with the second law of thermodynamics \cite{lucas2015, lucashartnoll2017}.  Section \ref{sec:VF} presents the appropriate theory for the variational formulation, and constructs homogenized transport coefficients (effective tensors) in a variety of contexts and boundary conditions. 

A second obstruction to standard homogenization of the model \eqref{eq:systWF} is that its variational counterpart is well-posed only when the coefficients are {\em sufficiently oscillatory}. In other words, there would be no homogenization theory for homogeneous coefficients.  There is an important physical reason why this must be the case. In the absence of any inhomogeneity and with suitable boundary conditions, there is an infinite number of solutions to (\ref{eq:systWF}), characterized by arbitrary constant shifts to the velocity $v$.  Physically, these shifts are allowed because momentum is conserved in the absence of inhomogeneity in $n$ and $s$: momentum conservation implies ballistic transport \cite{andreev,lucas2015}, and so (\ref{eq:introdiff}) would not make sense.  However, in the presence of inhomogeneity, there is no additional conservation law for momentum; only charge and energy are conserved.  Transport is again diffusive and the homogenized limit of (\ref{eq:systWF}) is of the form (\ref{eq:introdiff}).

A final remark is that the homogenization of (\ref{eq:systWF}) is only well-posed when $\eta>0$: higher derivative terms play a critical role in stabilizing the equations.  This subtlety makes generalizing Voigt bounds (which are elementary for standard homogenization) to the system (\ref{eq:systWF}) a non-trivial task.

With the current-based variational formalism in place, we extend a very selected subset of the standard results of homogenization to our degenerate setting. Section \ref{sec:per} generalizes Tartar's energy method to show standard weak/strong convergence results in the context of Dirac fluids. In section \ref{sec:rand}, we revisit the sub-additive properties that underline the homogenization results obtained by $\Gamma-$convergence in a wide variety of contexts \cite{dal2012introduction,dal1986nonlinear,muller1987homogenization}. We confine ourselves to using sub-additivity to show that homogenized coefficients converge in an appropriate sense as the size of the domain increases to effective deterministic coefficients when the random oscillations are stationary and ergodic. 

The path to proving more general standard homogenization results \cite{BLP-78,BMW-JRAM-94,dal2012introduction,JKO-SV-94} in the context of \eqref{eq:systWF} is reasonably clear but not considered further here. 
In the future, we hope that some of the mathematical tools developed here will generalize to the homogenization of a quantum kinetic theory describing interacting electrons moving through inhomogeneous media \cite{lucashartnoll2018}.

\section{Variational formulation and effective coefficients}
\label{sec:VF}

We now describe a general framework that encompasses \eqref{eq:systWF}.

Let $D$ be the spatial dimension, and $m$ be the number of conserved currents.
We study the following problem posed on a open bounded domain $X\subset\Rm^D$. 
 Let \[
  \psi \in L^2(X) \otimes \Rm^{1\times m},\ v \in H^1(X)\otimes\Rm^{D\times1},\quad a\in C^1(X)\otimes\Rm^{(m-1)\times m},\ b\in C^1(X)\otimes \Rm^{1\times m}.
\]
The abstraction (of a generalization) of (\ref{eq:systWF}) takes the form
\begin{equation}\label{eq:syst}
 \begin{array}{rcl}
  \nabla \cdot ( -(\nabla\psi) a^T a + vb ) &=& 0,
\\
  (\nabla \psi) b^T + Lv &=& 0,
\end{array}
\end{equation}
where 
the vector operator
\[
   L = - \nabla\cdot\eta \nabla - \nabla \zeta \nabla\cdot = S^T S
\]
has bounded coefficients $\eta\geq\eta_0>0$ (as a symmetric positive definite matrix) and $\zeta\geq0$  that may naturally be written as a sum of squares in the form $S^TS$ with \begin{equation}
    S=(\eta^{\frac12}\nabla,\zeta^{\frac12}\nabla\cdot)^T.\end{equation}
Beyond the positive definite constraint on $\eta$, our main assumption is that $a$ is of rank $m-1$ and that the span of $a$ and $b$ is of full rank $m$.  It is this rank deficiency that makes the problem interesting mathematically.   Note that $a$ is full rank for generic electron fluids \cite{lucasreview}: in the system (\ref{eq:systWF}) $a$ has reduced rank due to approximate Lorentz covariance.

To compare (\ref{eq:syst}) with (\ref{eq:systWF}), we set $m=2,$  $\psi=(\mu,T),$
and 
\begin{equation}
a = \sigma_Q^{\frac12} (-1,\gamma),\qquad  b=(n,s). \label{eq:abdiracfluid}
\end{equation}
The assumption that the span of $a$ and $b$ is of full rank is then that $\gamma n+s\ne 0.$
\subsection{Saddle point system}
Upon defining
\begin{equation}\label{eq:opAB}
  A\psi = (\nabla \psi) a^T,\quad B \psi = (\nabla\psi) b^T,
\end{equation}
the above equation is equivalent to
\begin{equation}\label{eq:galsaddle}
  \left(\begin{matrix}  S^TS & B  \\ B^T & -A^TA \end{matrix} \right) \left(\begin{matrix}  v\\ \psi\end{matrix} \right) = 0,
\end{equation}
which need to be augmented with appropriate boundary conditions and source terms. We consider two types of boundary conditions: either periodic boundary conditions for $(\psi,v)$ when $X$ is a torus, or Dirichlet boundary conditions imposing that $\psi$ and $v$ vanish on $\partial X$. The source terms of interest here involve linear profiles for $\psi$, i.e., $\psi=p\cdot x+\varphi$ with $\varphi$ either periodic or with vanishing Dirichlet boundary conditions.

With such boundary conditions and in appropriate topologies, which ensure that integration by parts are justified and boundary contributions vanish, we can write the above problem as a saddle point for the following functional
\begin{equation}\label{eq:saddle}
  \mF(v,\psi) = \frac12\|Sv\|^2 - \frac12 \|A\psi\|^2 + \aver{B\psi,v}.
\end{equation}
The above system may then be seen as the saddle point (Euler-Lagrange) equations associated to the above formulation, which may formally be obtained as $\min_v\max_\psi\mF$ or $\max_\psi\min_v \mF$. We do not try to prove a minimax result (the equality of the last two objects) as the above functional do not seem to satisfy standard results on saddle point theorems. In particular, while $Sv$ controls $v$ up to a constant vector field,
$A\psi$ is degenerate as $a$ is assumed to be only of rank $m-1$.

Instead of working with a saddle point system, and since our objective is to compute currents anyway, we consider a different functional setting and introduce the currents
\begin{equation}\label{eq:J}
  J =- (\nabla\psi) a^Ta + vb \in L^2(X)\otimes \Rm^{D\times m}.
\end{equation}
Note that the second line in the system of equations in \eqref{eq:galsaddle} is now simply of the form $\nabla\cdot J=0$. This respects the standard notion of currents as divergence-free vector fields. These vector fields are not independent since $Lv+B\psi=0$ couples them in a fashion we now analyze.

Since we assume that the span of $a$ and $b$ is of full rank $m$, 
we can introduce the dual or reciprocal matrices $w\in  C^1(X)\otimes \Rm^{(m-1)\times m}$ and
$u\in C^1(X)\otimes \Rm^{1\times m}$ (the rows of $a$ and $b$ are a basis and the rows
of $w$ and $u$ are a dual or reciprocal basis) such that
\begin{equation}\label{eq:dual}
w a^T =a w^T =I_{m-1},\qquad w b^T=b w^T=0,\qquad u a^T =a u^T=0,\qquad u b^T=b u^T=1.
\end{equation}
We then obtain by multiplying \eqref{eq:J} on the right by $w^T$ and then $u^T$ that
\begin{equation}\label{eq:mult}
-\nabla\psi a^T=Jw^T\quad\text{and}\quad v=J u^T.
\end{equation}
We next obtain from the momentum equation in \eqref{eq:syst} and from \eqref{eq:mult} above that
\begin{equation}\label{eq:mom}
-(\nabla \psi) b^T = Lv=LJ u^T.
\end{equation}
Since $w$ and $u$ are dual matrices to $a$ and $b,$ we have the identity
\begin{equation*}\label{eq:id}
I_m=a^T w+ b^T u,
\end{equation*}
so we get by multiplying $-\nabla\psi$ on the right by this identity and using \eqref{eq:mult}
and \eqref{eq:mom} that
\begin{equation}\label{eq:psi}
- \nabla \psi = J w^T w +L (J u^T) u = J w^T w  + S^T S (J u^T) u  =:\mA J.
\end{equation}
This shows that the curl of both sides vanishes, which provides a system of partial differential constraints on the divergence-free currents $J$. The above operator $\mA$ is formally linear.

\subsection{Variational formulation for currents}
Our objective is now to find a variational formulation for divergence-free currents $J$.  We need to integrate the last term by parts to construct a bilinear form for the currents $J$. 

We denote by $\aver{\cdot,\cdot}$ the usual (real) inner product on $L^2(X)\otimes\Rm^{m_1\times m_2}$, the space of square integrable functions with values in tensors in $\Rm^{m_1\times m_2}$ and use the same notation independently of $m_1,m_2$. We use the standard inner product $(\cdot,\cdot)$ for tensors in $\Rm^{m_1\times m_2}$, given by $(A,B)={\rm Tr}(A^TB)$.

We find, with $\tilde J$ a sufficiently smooth test function,
\[
  \aver{-\nabla\cdot\eta\nabla Ju^T,\tilde Ju^T} = \aver{\eta\nabla Ju^T,\nabla \tilde J u^T} - \dint_{\partial X} (\nu^T\eta \nabla Ju^T,\tilde J u^T) d\sigma 
\]
with $d\sigma$ the Lebesgue measure on $\partial X$ and $\nu$ the outward unit normal, and
\[
\aver{-\nabla\zeta\nabla\cdot Ju^T,\tilde J u^T} = \aver{\zeta\nabla\cdot Ju^T,\nabla\cdot \tilde J u^T} - \dint_{\partial X} (\nu^T \zeta \nabla\cdot Ju^T,\tilde J u^T) d\sigma.
\]
We observe that the boundary contributions vanish in three standard situations: (i) when $\tilde Ju^T=0$ on $\partial X$ (solution and test functions satisfying Dirichlet conditions); (ii) when $\nu^T(\eta\nabla Ju^T + \zeta\nabla\cdot Ju^T)=0$ (natural or Neumann boundary conditions); or (iii) when all above functions are periodic on a periodic domain $X$ (and arbitrary shifted copies). Such boundary conditions translate into corresponding (complicated) boundary conditions for $(\psi,v)$ via \eqref{eq:psi} and \eqref{eq:mult}.

Multiplying the equation for $\nabla \psi$ on the right by $\tilde J$ and integrating over $X$ with the above constraints on the  boundary, we find
\begin{equation}\label{eq:psiJ}
  -\aver{\nabla \psi, \tilde J} = \aver{Jw^T,\tilde Jw^T}  + \aver{S Ju^T,S \tilde Ju^T} =: \fa(J,\tilde J).
\end{equation}
The bilinear form $\fa(J,\tilde J)$ is formally non-negative and we will soon find conditions so that it is an inner product. The saddle point is therefore replaced by a minimization problem. The derivation is reminiscent of the duality between Lagrangian and Hamiltonian dynamics. A major difference is the presence of the operator $L$, which introduces higher-order derivatives than in standard Lagrangians and whose inverse is non-local.

The currents satisfy the constraint $\nabla\cdot J=0$. We therefore introduce $f$ as the anti-symmetric tensor (representing a differential form in $\Lambda^2(X)$) so that
\[
  J = \nabla\cdot f + c
\]
where $c$ is a matrix of constant coefficients and $f$ is chosen with Dirichlet or periodic boundary conditions. In coordinates, this is
\begin{equation}\label{eq:Jf}
  J_{lk}(x) = \partial_j f^{j}_{lk}(x) + c_{lk},\qquad f^j_{lk}+f^l_{jk}=0,\qquad 1\leq l,j\leq D,\quad 1\leq k\leq m.
\end{equation}
In dimension $D=2$ and for each index $1\leq k\leq m$, the antisymmetric tensor may be identified with a scalar function $f=f_k$ and $J_k=\nabla^\perp f_k+c_k$ with $\nabla^\perp f_k=(-\partial_2 f_k,\partial_1 f_k)^T$.

Homogenization is a macroscopic regime that aims to answer the following type of questions. When a profile $p=\aver{\nabla\psi}$ is prescribed macroscopically, what is the average current $c=\aver{J}$ generated? When a prescribed average current $c=\aver{J}$ flows, what is the average profile $p=\aver{\nabla\psi}$? 
In the rest of the section, we consider the functional setting allowing us to solve these two dual problems. 

In the first problem with prescribed macroscopic current $c=\aver{J}$, we wish to solve a problem of the form: Find $J=\nabla\cdot f$ such that 
\[
 \fa(c+J,\tilde J)=0
\]
holds for each admissible test function $\tilde J$ and for a solution $J$ with appropriate boundary conditions for $f.$ 
We first consider periodic and Dirichlet-type boundary conditions and show that the above problem admits a unique solution.


One of the main novelties of the proposed theory is the derivation of mesoscopic transport properties (conductivities) that are defined (finite) only provided that a sufficient amount of spatial inhomogeneity arises, in order to prevent the global ballistic transport of momentum ($v$).  
When all coefficients $b_j$ are functions that do not genuinely depend on the $D$ spatial dimensions, such ballistic transport will arise. We therefore need an assumption of sufficient oscillation in the coefficients $b_j$.

\begin{definition} [Oscillations]
 We define the oscillation of the coefficients $b$ as
\begin{equation}\label{eq:osc}
   \fO = \min_{\theta\in\Sm^{D-1}}  \max_{1\leq j\leq m}  \aver{(\theta\cdot\nabla b_j)^2}.
\end{equation}
\end{definition}
We also define spaces of smooth test functions as follows.
\begin{definition}[Test functions.] \label{def:testfunctions} For $X$ a periodic domain in $\Rm^D$, we define $\mD_\sharp$ as the space of currents $J=\nabla\cdot f$ where $f$ is a smooth antisymmetric tensor in the sense of \eqref{eq:Jf} with periodic boundary conditions.  

For $X$ a bounded open domain in $\Rm^D$, we define $\mD_D$ as the space of currents $J=\nabla\cdot f$ with $f$ a smooth antisymmetric tensor such that both $f=0$ and $J u^T=(\nabla\cdot f) u^T=0$ on $\partial X$.  

Finally, we denote by $\mD$ the space of smooth currents $J=\nabla\cdot f$ on $X$.
\end{definition}



Note that by the Hodge decomposition \cite[Chapter 5]{Taylor-PDE-1}, any periodic vector field may be decomposed as $\nabla\phi+\nabla\cdot f +c$ with $c$ harmonic and hence a constant vector. Since $J$ is divergence-free, $\phi=0$ and $\mD_\sharp$ concerns the periodic  solenoidal component. 

For both periodic and Dirichlet boundary conditions, we find that $\aver{c+J}=\aver{c}+\aver{\nabla\cdot f}=\aver{c}$ so that the average flux for $c+J$ is indeed $c$.


In the case of Dirichlet boundary conditions, integration by parts in the term $S^TSJu^T$ imposes that we use test functions $Ju^T=v$ vanishing on the boundary.

In both cases of currents in $\mD_\sharp$ or $\mD_D$, the passage from \eqref{eq:psi} to \eqref{eq:psiJ} is justified. We now prove that $\fa$ is well behaved when the coefficient $b$ is sufficiently oscillating.
\begin{lemma}\label{lem:a}
  Let us assume that $b$ is sufficiently oscillatory in the sense that for a constant vector $c$, then $c\cdot\nabla b_j=0$ for all $1\leq j\leq m$ a.e. implies that $c=0$. Then $\fa(J,\tilde J)$ is an inner product on $\mD_\sharp$ and $\mD_D$.
\end{lemma}
\begin{proof}
 The form is clearly symmetric, bilinear, and continuous on the above spaces equiped with their natural Fr\'echet topology. It remains to show that $\fa(J,J)=0$ implies that $J=0$. Here, $J$ is any smooth divergence-free vector field. We deduce that $Jw^T=0$ and $\nabla Ju^T=0$. This show that $Ju^T=c$ is a constant vector in $\Rm^D$. Recall that $J=Jw^T a + Ju^Tb$ so that $J=cb$. Therefore,
\[
  0=\nabla\cdot J = \nabla\cdot (cb) = c^T\nabla b. 
\]
This implies that $c=0$ by assumption and shows that $\fa(J,\tilde J)$ is an inner product both for Dirichlet and periodic boundary conditions.
\end{proof}
Note that $\fO>0$ is a sufficient condition for the assumption in the preceding lemma.
\begin{definition}[Hilbert spaces.] We denote by $\mH_\sharp$ and $\mH_D$ the completions of the pre-Hilbert spaces $\mD_\sharp$ and $\mD_D$, respectively, for the norm generated by $\fa(J,\tilde J)$. We denote by $\mH$ the completion of $\mD$ for the same inner product.

While constant tensors in $\Rm^{D\times m}$ belong to $\mH$, they do not belong to $\mH_\BC$ for $\BC=D,\sharp$. We introduce $\check\mH_\BC=\Rm^{D\times m}\oplus \mH_\BC$.
\end{definition}
With this notation, we have the following Riesz representation theorem:
\begin{corollary}[Riesz] Let $S$ be a continuous linear form from $\mH_\BC$ to $\Rm$. Then the following problem
\begin{equation}\label{eq:varform}
  \mbox{ Find }\ J \in \mH_\BC \ \mbox{ such that } \  \fa(J,\tilde J)=S(\tilde J) \ \mbox{ for all } \tilde J\in \mH_\BC
\end{equation}
admits a unique solution.
\end{corollary}
The following result shows that $\fa$ is coercive on $\check\mH_\BC$ with a stability controlled by $\fO$.
\begin{proposition}\label{prop:stab}
  There exists a constant $C$ that depends on $X$ and bounds on the coefficients $(u,v)$ such that
\begin{equation}
  \aver{|J|^2} \leq \frac{C}{\fO} \fa(J,J) \quad \mbox{ for all } J\in \check\mH_\BC. \label{eq:l2bd}
\end{equation}
\end{proposition}
Here, $\Rm^{D\times m}$ is the space of constant currents $J=c\in\Rm^{D\times m}$.
\begin{proof}
By construction, $\fa(J,J)$ controls $Jw^T$ and $\nabla Ju^T$. A standard Poincar\'e inequality shows that $\fa(J,J)$ controls $Ju^T-c$ as well with $c=\aver{Ju^T}$. This means that $J=Jw^Ta+(Ju^T-c)b+cb$ with $cb$ the only term that is not controlled in the $L^2$ sense yet, i.e., $\|J-cb\|^2 \leq C \fa(J,J)$. So, far, this estimate holds for any vector field $J$, not necessarily divergence-free ones. Let us consider one component $b_j$.  For any smooth, compactly supported $\varphi$ in the open set $X$, we therefore have $|\aver{J_j-b_jc,\nabla\varphi}|\leq C \fa(J,J)$ from the above and  find that $\aver{\nabla \varphi,J_j}=0$ for a (now) divergence-free field $J$ so that $\aver{\nabla \varphi,cb_j}=-\aver{\varphi,c\cdot\nabla b_j}$ is controlled by $\fa(J,J)$. Choosing $\varphi=\chi (c\cdot\nabla b_j)$ (or a regularized version of this term) with $\chi$ smooth compactly supported in $X$ and converging to $\one_X$, we find, with $\hat c=c/|c|$ (assuming $|c|>0$ otherwise we are done) that
\[
  |c|^2 \aver{(\hat c\cdot\nabla b_j)^2} \leq C \fa(J,J).
\]
The above holds after maximizing over $1\leq j\leq m$ and has to hold for all orientations $\hat c$ of $c$. Employing (\ref{eq:osc}), we complete the proof.
\end{proof}
The control of $\fa(J,J)$ is therefore equivalent to that of $|J|^2$ and $|\nabla Ju^T|^2$ for divergence-free vector fields. We cannot expect any a priori energy control for $\nabla Jw^T$. In terms of the original variables $\psi$, this shows that $\nabla\psi a^Ta$ should be bounded in the $L^2$ sense while $\nabla\psi b^T=-L Ju^T$ is only bounded in the $H^{-1}$ sense.

\subsection{Effective tensors}
Let $c$ be a fixed tensor in $\Rm^{D\times m}$. The map
\[
  \tilde J \mapsto \fa(c,\tilde J)
\]
is clearly continuous from $(\mH_\BC,\fa)$ to $\Rm$ using the smoothness assumptions on the coefficients $u$ and $v$ and applying the Cauchy-Schwarz inequality. As a consequence, by the Riesz representation theory, there is a unique solution in $J\in\mH_\BC$ to the problem
\begin{equation}\label{eq:EL}
  \fa(c+J,\tilde J)=0, \qquad \forall \tilde J\in \mH_\BC.
\end{equation}
Moreover, the solution $J\in \mH_\BC$ is linear in $c$ so that
\begin{equation}\label{eq:defbara}
  \fa(c+J,c+J) = |X| (c,\bar \fa c)
\end{equation}
is quadratic in $c$ for some effective linear symmetric map $\bar\fa$ from $\Rm^{D\times m}$ to $\Rm^{D\times m}$, which we may interpret as a tensor $\bar\fa \in \mathbb{R}^{D\times m \times m \times D}$.  We recall that $(c,d)={\rm Tr}( c^Td)$ is the standard inner product on $\Rm^{D\times m}$. Here and below, we denote by $|X|$ the (Lebesgue) volume of the domain $X$. 

The tensor $\bar\fa=\bar\fa_{\BC}$ depends on the domain $X$ and the boundary conditions $\BC=D,\sharp$. Moreover, we deduce from Proposition \ref{prop:stab} that
\begin{equation}\label{eq:bda}
 (c, \bar \fa c) =\frac{1}{|X|}\fa(c+J,c+J)\geq C \|c+J\|_2^2 \geq C  \|c\|^2
\end{equation}
for a positive constant $C$ since $\nabla\cdot f$ and $c$ are orthogonal in the $L^2$ sense (by the Hodge decomposition). This stability result shows that the tensor $\bar \fa$ is positive definite.


The effective tensor $\bar\fa$ is the solution to the following minimization problem:
\begin{lemma}\label{lem:minim}
  The tensor $\bar \fa$ constructed in \eqref{eq:defbara} obeys 
\begin{equation}\label{eq:minim}
   (c, \bar \fa c) = \min_{J\in \mH_b} \frac1{|X|} \fa(c+J,c+J).
\end{equation}
\end{lemma}
\begin{proof}
  Since $\fa$ is a bilinear form, the solution of $\fa(J,\tilde J)=S(\tilde J)$ for all $\tilde J\in\mH_b$ is also the unique solution to the minimization problem
\[
   \min_{J\in \mH_b} \frac12 \fa(J,J)-S(J).
\]
Applying this to $S(J)=-\fa(c,J)$ shows that $J$ is also the solution to the following minimization
\[
   \min_{J\in \mH_b} \frac12 \fa(c,c) +  \frac12 \fa(J,J) + \fa(c,J) =  \min_{J\in \mH_b} \frac12 \fa(c+J,c+J).
\]
But evaluated at $J$ satisfying \eqref{eq:EL}, this is precisely $\frac{|X|}{2} (c, \bar \fa c)$.
\end{proof}
The two effective tensors $\bar \fa_\BC$ obtained  with boundary conditions $\BC=\sharp$ and $\BC=D$ may therefore both be obtained as the minimal (energy) value of a minimization problem. There is no reason to expect both tensors to be equal. However, since $f\in\mH_D\implies f\in \mH_\sharp$, then $\bar\fa_D\geq\bar\fa_\sharp$ in the sense of symmetric tensors.  

\subsection{Small oscillations}\label{sec:small}
We now show that oscillations in the coefficients are crucial to obtain non-degenerate diffusion.  Let us assume that all coefficients $a$ and $b$ are given by $a=a_0+\lambda a_1$, $b=b_0+\lambda b_1$, with $a_0,b_0$ constant and non-vanishing, $a_1$ and $ b_1$ oscillatory but smooth coefficients, and $\lambda \ll 1$ sufficiently small: $a$ and $b$ are nearly, but not exactly, constant. Let $w_0$ be the corresponding constant coefficient, defined analogously to (\ref{eq:dual}). Then we have the following result.
\begin{proposition}\label{prop:smallosc}
 Under the above hypotheses, the effective tensors $\fa_\BC$ have $D(m-1)$ eigenvalues of order $O(1)$ and $D$ eigenvalues of order $O(\lambda^2)$ in the sense that for any tensor $c$ such that $c w_0^T=0$ then $(c,\fa c) = O(\lambda^2)$ while for any tensor $c$ orthogonal to that set of dimension $D$, then $(c,\fa c)=O(1)$.
\end{proposition}
Here, we mean $x=O(y)$ if $C^{-1}x\leq y\leq Cx$ for some constant $C>0$ independent of $y$.
\begin{proof}
 The oscillation coefficient $\fO$ is clearly of order $\lambda^2$. The construction of $\fa$ and the results of Proposition \ref{prop:stab} show that for some constant $C$ independent of $J$ and $\fO$, we have
 \[
  C^{-1}\left(\|Jw^T\|^2+\|\nabla Ju^T\|^2 + \fO \aver{|J|^2} \right)\leq \fa(J,J) \leq C( \|Jw^T\|^2+\|\nabla Ju^T\|^2) .
 \]
 Here, $J$ is any admissible current. The homogenized coefficient $(c,\bar\fa c)$ obtained as a minimization is smaller than the variational formulation obtained for $J=c$. By assumption on the coefficients, this means that $(c,\bar\fa c)\leq \|c w_0^T\|^2 + O(\lambda^2)$, which provides the upper bound.
 
 To derive the lower bound, let $J=\nabla\cdot f+c$ realize the minimum of the optimization problem (\ref{eq:minim}). Consider the term
 \[
   \|Jw^T\|^2=\aver{(\nabla\cdot f+c)w^T,(\nabla\cdot f+c)w^T} = \aver{|c w^T|^2} + \aver{|(\nabla\cdot f)w^T|^2} + 2\aver{\nabla\cdot f w^T, cw^T}.
 \]
 By integration by parts, the last term is of order $O(\lambda)$, which shows that $(c,\bar\fa c)$ is bounded below by $|cw_0^T|^2$ up to a negligible term, and hence is of order $O(1)$ when that last term is. Let us now assume that $cw_0^T=0$. Then, Proposition \ref{prop:stab} or the above estimate shows that $(c,\bar\fa c)$ is bounded below by $\fO\aver{|\nabla\cdot f+c|^2}$, itself bounded below by $\fO|c|^2$ up to a multiplicative constant.
\end{proof}
The above result shows that oscillations in the coefficients $b$ are necessary to obtain a finite current $\aver{J}$ of order $\lambda^{-2}$ for an average linear profile such that  $\aver{\nabla\psi}$ is of order $O(1)$.


%
%
\subsection{Natural boundary conditions}
%
%
We now consider a dual minimization problem associated with natural (Neumann) boundary conditions. The constants $c$ above were imposed on the currents. As we saw, the solution of the problem may then be interpreted as a resulting macroscopic linear displacement profile for the terms $\psi$. Alternatively, we may impose a linear profile on $\psi$ and deduce an average current. This is obtained by looking at the following minimization problem, which we recast as a maximization problem by sign change. Let $p$ be a constant tensor in $\Rm^{D\times m}$, which we want to interpret as an average profile for $-\nabla\psi$. We then solve the problem:
\begin{equation}\label{eq:varp}
  \mbox{ Find } J \in \mH \ \mbox{ solution of } \quad \max_{J\in \mH} -\frac12 \fa(J,J) + \aver{p,J}.
\end{equation}
This problem is equivalent to its Euler-Lagrange version, which takes the form
\begin{equation}\label{eq:ELp}
  \mbox{ Find } J \in \mH \ \mbox{ solution of } \quad  \fa(J,\tilde J) =\aver{p,\tilde J},\quad \forall \tilde J\in \mH.
\end{equation}

The map $J\to\aver{p,J}$ is continuous on $(\mH,\fa)$ when $\fO>0$ by the Cauchy-Schwarz inequality, so that by the Riesz representation theorem, the above equation admits a unique solution in $\mH$, which is also given by the above maximization problem. This shows that the maximum is obtained for an energy
\[
  -\frac12 \fa(J,J) + \aver{p,J} = \frac12 \fa(J,J)  = \frac{|X|}2 (p, \bar\fb p)
\]
for a constant tensor $\bar\fb$ since $J$ is linear in $p$.

The above homogenized coefficient may also be computed with periodic boundary conditions. This requires us to introduce the product space $\check\mH_\sharp=\Rm^{D\times m}\times\mH_\sharp$ of constant tensors $c$ and $J\in\mH_\sharp$. We then state the maximization problem as
\begin{equation}\label{eq:varpper}
  \mbox{ Find } (c,J) \in \check\mH_\sharp \ \mbox{ solution of } \quad \max_{(c,J)\in \check\mH_\sharp} -\frac12 \fa(c+J,c+J) + \aver{p,c}.
\end{equation}
Note that $\aver{p,J}=0$ for all $J\in\mH_\sharp$.
Its Euler-Lagrange equation is
\begin{equation}\label{eq:ELpper}
  \mbox{ Find } (c,J) \in \check\mH_\sharp \ \mbox{ solution of } \quad  \fa(c+J,\tilde c+  \tilde J) = \aver{p,\tilde c},\quad \forall (\tilde c,\tilde J)\in \check\mH_\sharp .
\end{equation}

We then find as for the preceding problem in \eqref{eq:EL} with periodic boundary conditions that
\[
   \fa(c+J,c+J)  = {|X|} (p, \bar\fb_\sharp p) =  \aver{p,J} = \aver{p,c}. 
\]

We have by now obtained four different homogenized coefficients for equations on a fixed periodic domain $X$ with periodic, Dirichlet, or natural boundary conditions. 
\begin{proposition}\label{prop:comp}
 We have the following ordering:
 \begin{equation}\label{eq:compcoefs}
     \bar\fb^{-1}\leq\bar\fb_\sharp^{-1}=\bar \fa_\sharp\leq \bar \fa_D.
 \end{equation}
\end{proposition}
\begin{proof}
  All coefficients are positive definite operators from $\Rm^{D\times m}$ to itself and hence invertible as we already saw.
 Two relations are clear from the minimization procedures. Since functions with vanishing Dirichlet conditions (recall Definition \ref{def:testfunctions}) are also periodic, then $\bar \fa_\sharp\leq \bar \fa_D$. Since arbitrary functions in $\mH$ have more general gradients than $c+\nabla\cdot f$ for $\nabla\cdot f\in \mH_\sharp$, then $\bar\fb\geq\bar\fb_\sharp$. 
 
 It remains to show the middle equality. Consider the map from $c$ to $p=\bar\fa_\sharp c$. We have for all $(\tilde c,\tilde J)\in \check\mH_\sharp$ that
 \begin{equation}\label{eq:relpac}
     \fa(c+J,\tilde c+\tilde J)= \fa(c + J,\tilde c) = \aver{p,\tilde c},
 \end{equation}
 where the first equality comes from the Euler-Lagrange constraint $\fa(c+J, \tilde J)=0$ in \eqref{eq:EL} and where $p$ comes from the Riesz representation since $\tilde c\mapsto \fa(c+\tilde J,\tilde c)$ is a continuous linear map. The above $p$ is the one giving rise to the homogenized coefficient $\bar\fa_\sharp$ since by construction
 \[
   \fa(c+J,c+J) =\aver{p,c}= {|X|}(c,\bar\fa_\sharp c).
 \]
 Now, \eqref{eq:relpac} is also the defining equation for $c=c(p)$ in \eqref{eq:ELpper}. This shows that $\aver{p,c}={|X|}(p,\bar\fb_\sharp p)={|X|}(c,\bar\fa_\sharp c)$ and this concludes the derivation.
\end{proof}


We do not expect equality in the above relations as boundary conditions influence the effective coefficient. These coefficients become asymptotically equal in the setting of periodic (or random) coefficients as the size of the domain increases and boundary effects become negligible. We will (partially) revisit this question in subsequent sections.

A notable feature of the Dirichlet and natural boundary conditions is that the homogenized coefficients enjoy a subadditive property. The minimization with Dirichlet boundary conditions leading to the calculation of $\bar\fa_D$ involves more test functions as the domain's size increases so that its value decreases (is subadditive on average) with the domain's size. The maximization with natural boundary conditions leading to $\bar\fb$ involves less test functions as the domain size increases so that its value also decreases (is also subadditive on average) with the domain's size. This means that $\bar\fa=\bar\fb^{-1}$ then increases in the same sense. We will come back to these features when discussing the setting of stationary random coefficients.


%
%
\subsection{Variational framework for original variables}

Once $J$ is uniquely determined, we can come back to the original variables $\psi$ and $v$. We have $v=Ju^T$ unambiguously defined. We wish to show that $\nabla\psi$ in \eqref{eq:psi} is indeed a gradient. For this, we observe that it is characterized by
\[
 - \aver{\nabla\psi,\tilde J}=\fa(J,\tilde J)= 0,\qquad \forall \tilde J\in \mH_b.
\]
The case $D=1$ is somewhat singular as the spaces $\mH_\BC=\{0\}$ are trivial. We consider $D=1$ in the next section and assume for the rest of the section that $D\geq2$.
Let $J_\psi\equiv \nabla \psi$. By the Hodge decomposition either on a periodic domain or on an open domain \cite[Chapter 5]{Taylor-PDE-1}, we may write it as $\nabla\psi+\nabla\cdot f$ and deduce from the above constraint that $\nabla\cdot f=0$ so that $J_\psi$ is indeed a gradient field. When $X$ is the periodic unit cube, then the Hodge decomposition shows that $\nabla\psi$ may be written as $\nabla\tilde\psi+c$ with $\tilde\psi$ periodic and $c$ a harmonic form on the torus, i.e., a constant vector.

To obtain uniqueness of a solution, as well as relate effective, homogenized coefficients, we need to introduce a natural functional setting for $\nabla\psi$. Let us introduce the (topological) dual spaces $\mH^*_N=(\check\mH_D)^*$ and $\mH_\sharp^*=(\check\mH_\sharp)^*$, which we summarize as $\mH^*_\CB=(\check\mH_\BC)^*$. Symmetrically, we can add constant vectors to the dual spaces by introducing $\check\mH^*_\CB=(\mH_\BC)^*$.  By the Riesz representation theory, duals to Hilbert spaces are isometric to the Hilbert spaces. Moreover, the isometry may be realized by a linear map since our inner products are real-valued, and this linear map is given by the operator $\mA$ in \eqref{eq:psi} since
\[
  \fa(J,\tilde J) = \aver{\mA J,\tilde J}
\]
with the right-hand side the duality product between the Hilbert spaces. This shows that $-\nabla\psi=\mA J$ with $\nabla\psi\in\mH^*_\CB$ with the appropriate (inherited) boundary conditions.

We have therefore constructed a solution $(\psi,v)$ with $\nabla\psi\in\mH^*_\CB$ and $v=Ju^T$. Let us assume the existence of two such solutions. The difference of such solutions would solve the equation with no source terms. Since integrations by parts in \eqref{eq:galsaddle} on $X$ are allowed by construction of the functional spaces, we deduce that $Sv=0$ and $A\psi=0$. This implies that $v$ is constant. This in turns implies that $B\psi=0$. But $A\psi=0$ and $B\psi=0$ implies that $\nabla\psi=0$. Equation \eqref{eq:syst} now implies that $\nabla\cdot vb=0$ so that $v=0$ by the oscillation property, and hence the uniqueness of the solution $(\nabla\psi,v)$. To ensure that $\psi$ is uniquely defined, we can prescribe its value at one point, or its average over $\partial X$. The detour by a variational formulation for the currents therefore provides an existence and uniqueness result for the system \eqref{eq:syst}.

In both functional settings, we have $-\nabla\psi=\mA J$ so that 
\[
  -\aver{\nabla\psi,c} = \aver{\mA J,c} = \fa(c+J,c+J) = |X| (c,\bar\fa c) = -(\aver{\nabla\psi},c).
\]
This shows that we obtained what we were aiming for, namely a relation of the form
\begin{equation}\label{eq:homog}
    -\aver{\nabla \check\psi} = \bar \fa \aver{\check J} \in \Rm^{D\times m},
\end{equation}
with $\check\psi$ the sum of a linear profile and a component satisfying prescribed boundary conditions (so that $\aver{\nabla\check\psi}=\aver{p+\nabla\psi}=p$) and $\check J=\nabla\cdot \check f$ for $\check f$ a divergence free field also written as the sum of a linear profile and a component satisfying prescribed boundary conditions (so that $\aver{\check J}=\aver{c+J}=c$).

The interpretation for the negative sign above is that linear profiles for $\check\psi$ increasing in one direction generate currents propagating in the opposite direction, which is physically expected.

The stability result \eqref{eq:bda} shows that $\bar\fa$ is bounded below by a positive constant (as a symmetric tensor). This shows its invertibility and the fact that
\[
 \aver{\check J} = -\bar \fa^{-1} \aver{\nabla\check\psi}
\]
if we want to compute the average currents associated to linear profiles imposed on the variables $\check\psi$.

\medskip

The above derivation computes $\aver{\nabla\psi}$ from knowledge of an average current $c$. In order to compute $\aver{J}$ from knowledge of an average gradient $p$, we look for a variational formulation directly for $\nabla\psi$. As we saw in the above derivation, $J=-\mA^{-1}\nabla\psi$, which implies investigating the inverse of $\mA$.

Restricting ourselves to the periodic setting, we derive such a variational formulation for the variable $\psi$ and relate it to the variational formulation for the currents. 

We recall the equations
\[
   Lv=-\nabla\psi b^T,\qquad J= -\nabla\psi a^Ta + vb,\quad \nabla\cdot J=0. 
\]
The objective is to find $\psi$ such that $\nabla\psi = p + \nabla\tilde\psi$ with $\tilde\psi$ periodic. 

We first eliminate $v$ from the above equations. By the Fredholm alternative, this requires the condition
\[
  \aver{\nabla\psi b^T}=0 \in \Rm^{D\times1}
\]
and we then find that 
\[
  v= -L^{-1}(\nabla\psi b^T) + v_0,\qquad v_0\in \Rm^{D\times1}.
\]
After elimination, the equation for $\psi$ is therefore the following constrained system of equations
\[
   -\nabla\cdot  \Big( \nabla\psi  a^Ta + L^{-1}(\nabla\psi b^T) b + v_0b\Big)=0,\qquad \aver{ \nabla\psi b^T}=0.
\]
This may be recast formally with $\mA^{-1}\nabla\psi:=\nabla\psi  a^Ta + L^{-1}(\nabla\psi b^T) b$ as
\[
   -\nabla\cdot  \Big( \mA^{-1}\nabla\psi+ v_0b\Big)=0,\qquad \aver{ \nabla\psi b^T}=0.
\]
We wish to replace the constraint for $\nabla\psi$ by an equivalent one for test functions. Clearly $\aver{(p+\nabla\tilde\psi)b^T}=0$ shows that the above constraint does not hold for $\tilde\psi$ (yet). We thus look for one possible solution $\psi_0$ periodic such that 
\[
  \aver{(p+\nabla\psi_0) b^T}=0.
\]
Finding a solution requires that $b$ oscillate sufficiently for otherwise, we get $\aver{pb^T}=0$ since $\aver{\nabla\psi_0}=0$ by the periodicity assumption. 
The construction of $\psi_0$ requires the oscillatory constraint on $b$ and goes as follows. We decompose $p=\sum_{i=1}^D e_i\otimes p_i$ for $p_i$ arbitrary $m-$(co)vectors. By linearly, we construct a term $\psi_0$ for each $p=e_i\otimes p_i$. The above constraint is therefore
\[
  e_i \aver{p_ib^T} + \aver{\nabla\psi_0 b^T}=0.
\] 
Let $j$ be an index such that $b_j$ genuinely oscillates in all directions. We then set $\psi_{0k}=0$ for all $k\not=j$ and, to simplify notation, still denote by $\psi_0$ and $b$ the components $\psi_{0k}$ and $b_k$. We choose $\psi_0(x)=\varphi(x_i)$ so that the above constraint becomes
\[
  0= \aver{p_ib^T}  + \aver{\varphi'(x_i) \bar b_i(x_i)}
\]
where we denoted by $\bar b_i(x_i)$ the average of $b$ in all variables but $x_i$. It remains to choose 
\[
  \varphi(x_i) = \alpha_i (\bar b_i)'(x_i),\qquad \alpha_i = \dfrac{\aver{p_ib^T}}{\aver{|\bar b_i'|^2}},
\]
where the above denominator does not vanish by the oscillation assumption. 


Once $\psi_0$ has been chosen, we look for solutions of the form
\[
  \psi = x^Tp + \psi_0 + \tilde\psi, \qquad \nabla \tilde\psi \in \mG_\sharp,
\]
where the Hilbert space $\mG_\sharp$ is the completion for the inner product
\[
  \fb(G,G) = \aver{\mA^{-1}G,G}
\]
of smooth vector fields of the form $G=\nabla\tilde\psi$ with $\tilde \psi$ periodic and satisfying the constraint $\aver{ \nabla\tilde\psi b^T}=0$. That the above is an inner product is clear since $\fb(G,G)=0$ implies that $\mA^{-\frac12}G=0$ and hence $G=0$. 

The variational formulation of the above problem is therefore: 
\[
 \mbox{ Find } \nabla\tilde\psi \in \mG_\sharp \mbox{ such that } \fb(p+\nabla\psi_0+\nabla\psi,\nabla\check\psi) =0 \quad \mbox{ for all } \nabla\check\psi\in\mG_\sharp.
\]
We note that the term involving $v_0$ cancels in the formulation by the constraint on $\tilde \psi$. The Euler-Lagrange equations of the above variational formulation are again the above partial differential equation, where $v_0$ is the Lagrange multiplier associated to the constraint $\aver{\nabla\psi b^T}=0$.

This problem can then be recast as the following minimization 
\[
  \min_{\nabla\tilde\psi\in\mG_\sharp} \fb(p+\nabla\psi_0+\nabla\tilde\psi,p+\nabla\psi_0+\nabla\tilde\psi) = (p,\bar\fa_\sharp^{-1}p).
\]
Note that rather than a constrained minimization problem, we could also relax the averaging constraints in $\mG_\sharp$ and add to the Lagrangian a term of the form $(V,\aver{\nabla\psi b^T})$ for $V$ a vector of Lagrange multipliers.

\section{Bounds on homogenized coefficients}
\label{sec:WF}

We now describe some elementary estimates of $\bar\fa$.  Upper bounds on $\bar\fa$ are easily obtained and generalize the Voigt bounds for diffusion equations \cite{JKO-SV-94}: 
\begin{corollary} \label{cor:1d}
In the system (\ref{eq:syst}) with $a$ of rank $m-1$, \begin{equation}
(c,\bar\fa c) \le \left\langle \lVert cw^T \rVert^2 +  \lVert \nabla c u^T \rVert^2  \right\rangle .  \label{eq:uppera}
\end{equation}
Note that the derivative only acts on $u^T$ in the second term.  In $D=1$, (\ref{eq:uppera}) is saturated.  In $D>1$, if the background has small oscillations as defined in Section \ref{sec:small}, then (\ref{eq:uppera}) is asymptotically exact as $\lambda\rightarrow 0$.
\end{corollary}
\begin{proof}
 A straightforward application of Lemma \ref{lem:minim} proves the inequality. 
 
 \ml{For} $D=1$, the spaces $\mathcal{H}_{\mathrm{b}} = \lbrace \emptyset\rbrace$, so $J=c$ is the only allowable trial current.  By construction, it must be optimal.
 
 \ml{For} $D>1$, if $\lambda=0$, \ml{then} (\ref{eq:uppera}) gives $(c,\bar\fa c) = 0$ for the $D$ components of $c$ obeying $cw_0^T = 0$, and $(c,\bar\fa c) = O(1)$ for the remaining $D(m-1)$ components.   Since $\bar\fa$ is smooth in $\lambda$, by Proposition \ref{prop:smallosc} the $D(m-1)$ large coefficients of $\bar\fa$ are given by $(c,\bar\fa c) = \lVert c w_0^T \rVert^2 $.  For the remaining $D$ components of $c$, let $w=w_0 + \lambda w_1 + \cdots$ and $u = u_0 + \lambda u_1 + \cdots$.
If $Ju_0^T$ is not constant, the second term in (\ref{eq:uppera}) is $O(1)$;  if $Jw_0^T$ is spatially fluctuating but zero-mean, the first term in (\ref{eq:uppera}) is $O(1)$.  Hence the optimum must be $J=c+\lambda J_1 + O(\lambda^2)$.  At $O(\lambda^2)$, \[
\fa(J,J) \le C\lambda^2 \left\langle \lVert cw_1^T \rVert^2 +  \lVert \nabla c u^T_1 \rVert^2 + \lVert J_1w_0^T \rVert^2 + \lVert \nabla J_1 u^T_0 \rVert^2 \right\rangle .
\]
Since the four contributions above are additive, the minimum occurs when $J_1=0$.
\end{proof}

Observe that (\ref{eq:uppera}) depends on both $w$ and $\nabla u$:  although the system (\ref{eq:syst}) is of mixed order, it is generally required to have each term non-vanishing.


The $D$ $O(\lambda^2)$ eigenvalues of $\bar\fa$ have been previously computed in the physics literature and are proportional to the spectral weight of hydrodynamic correlation functions in many-body (quantum) systems \cite{andreev, lucas2015, lucashartnoll2017}.


\subsection{Dirac fluid}
As an example, we now return to the Dirac fluid equations (\ref{eq:systWF}).   
The coefficients $a$, $b$, $u$ and $w$ defined above are given by \begin{equation*}
   \begin{array}{ll}
       a = \sigma_Q^{\frac12} (-1,\gamma)  &  b=(n,s)  \\
       u=\frac1{\gamma n+s} (\gamma,1) & w=\frac1{\sigma_Q^{\frac12}(\gamma n+s)} (-s,n).
   \end{array} \label{eq:diracfluid}
\end{equation*}
We verify that $w^Ta+u^Tb=I_2$, with $\gamma n+s\geq c_0>0$ bounded away from $0$ as a condition for the uniform linear independence of $a$ and $b$.   Note that $c_0>0$ follows from thermodynamic inequalities independently of our transport calculation.

 Let us consider the one dimensional setting on $X=(0,1)$ with periodic boundary conditions (to simplify).  From (\ref{eq:uppera}) we obtain
\[
   -p = \bar \fa c, \qquad \bar \fa = \aver{\nu (u')^T(u') + w^T w} = \left(\begin{matrix} \aver{\nu u_1'u_1'+w_1^2} & \aver{\nu u_1' u_2' + w_1w_2} \\  \aver{\nu u_1' u_2' + w_1w_2} &  \aver{\nu u_2'u_2'+w_2^2}  \end{matrix} \right).
\]
We may easily confirm that there is no general relationship between $\kappa$ and $\sigma$.  The experimentally measured thermal conductivity is measured in the absence of charge current ($J_1=0$); since 
\[\left(\begin{array}{cc} \sigma &\ \alpha \\ \tilde\alpha &\ \tilde\kappa \end{array}\right)=\bar\fa^{-1},\]
it is given by 
\begin{equation}
    \kappa = \tilde\kappa-\tilde\alpha\sigma^{-1}\alpha=\left(\bar\fa_{11} - (\bar\fa_{12})^2\bar\fa_{22}^{-1}\right)/\det{\bar\fa}, \label{eq:truekappa}
\end{equation}
while $\sigma =\bar\fa_{22}/\det{\bar\fa}$. The Lorenz ratio is then \begin{equation*}
    \mathcal{L} = \frac{\kappa}{\sigma}=\det (\bar\fa)/\bar\fa_{22}^2.
\end{equation*}

\subsection{Galilean symmetry}
  Another simple scenario arises when the hydrodynamic transport equations are Galilean invariant \cite{andreev}.  In this case, one replaces $a$ in (\ref{eq:abdiracfluid}) with \[
a = \kappa_Q^{1/2}(0,1)
\]
while leaving $b$ unchanged.  Note that in this case, the charge current $J$ is identically proportional to $v$. 

It may be the case at ultra low temperatures, with sufficiently oscillatory $b$, that it is acceptable to ignore $T$, and think of the case $m=1$.  Here $a=0$, $u=b^{-1}=n^{-1}$ and $J=nv$ so that $v=bJ$ and $-\nabla\mu=u L(uJ)$. We then obtain the variational formulation $\aver{\nabla\psi,\tilde J} = \fa(J,\tilde J)$ with $\fa(J,\tilde J)=\aver{S Ju,S \tilde Ju}$. This generates an inner product provided that $n=b$ is sufficiently oscillatory as in the above example. In the one-dimensional setting, we would then obtain an average diffusion coefficient given by $\bar\fa=\aver{\nu (u')^2}$, which shows that $\nu>0$ and $u$ oscillating is necessary to obtain an effective coefficient such that $-\aver{\mu'} = \bar\fa \aver{J}$. The above theory, as well as the homogenization theory in the next two sections, also apply to this case.



%
\section{Scaling and periodic homogenization}
\label{sec:per}
%


This section considers transport over large distances compared to the scale at which the coefficients $a$ and $b$ oscillate. When the fluctuations are sufficiently stationary, i.e.,  independent under `macroscopic' spatial translations, we expect the homogenized coefficients obtained in earlier sections to dictate large scale transport. This is the so-called homogenization regime, or effective-medium regime. There is a huge mathematical literature on the analysis of such problems in a variety of contexts. The simplest non-trivial setting to analyze the emergence of such macroscopic models assumes that the coefficients are spatially periodic of period $1$ (after appropriately rescaling $x$): $\theta(x+\tau)=\theta(x)$ for each $\tau\in\Zm^D$ and $x\in\Rm^D$.

We wish to consider transport over a large domain $X_N$, with volume of order $N^D$ for $N\gg1$, and define $\eps=\frac 1N$. For a solution $J$ on $X_N$, we introduce $J_\eps(x)=J(Nx)$ posed on a domain $X$ independent of $N$, and after rescaling, which simply amounts to replacing every differentiation $\nabla$ by its rescaled version $\eps\nabla$, find the $\eps-$dependent bilinear form:
\begin{equation}\label{eq:aeps}
  \fa_\eps(J,\tilde J) = \aver{J w_\eps^T,\tilde J w_\eps^T} + \eps^2 \aver{S_\eps Ju_\eps^T,S_\eps \tilde Ju_\eps^T},
\end{equation}
where all coefficients $a_\eps$, $b_\eps$, $w_\eps$, $u_\eps$, as well as $\eta_\eps$ and $\zeta_\eps$ defining $S_\eps$ are given by $\rho_\eps(x)=\rho(\frac x\eps)$. We assume that the coefficient are such that  $\fa_\eps$ is uniformly (in $\eps$) an inner product on $\mH$. What we mean by this is that
\[
  \fO_\eps = \min_{\theta\in \Sm^{D-1}} \max_{1\leq j\leq m} \aver{(\theta\cdot \eps\nabla b_{\eps j})^2}
\]
is bounded below by a positive constant independent of $\eps$. Here, $\aver{\cdot,\cdot}$ is the $L^2(X)$ inner product. Note that if $X$ is the unit cube, then by periodicity $\fO_\eps$ is independent of $\eps$ and given by \eqref{eq:osc}. With these hypotheses, we have
\begin{equation}\label{eq:bdaeps}
   \fa_\eps(J,J) \geq \dfrac{C}{\fO_\eps} \aver{|J|^2}
\end{equation}
independent of $\eps$. This result is obtained on the domain $X_N$ seen as $N^D$ copies of the unit cube. On each of the cubes composing $X_N$, we apply Prop. \ref{prop:stab} and \eqref{eq:l2bd}, with a constant independent of the cube by periodicity of the coefficients (or by assumption that the oscillation is bounded below on each cube). After rescaling on $X$, we find \eqref{eq:bdaeps}.

During the rescaling process, $S$ is replaced by $\eps S_\eps$, since $\nabla$ is replaced by $\eps\nabla$. Recalling that we write $J = \nabla\cdot f$ (where $x$-linear coefficients are allowed in $f$), after rescaling, the spatially varying component $f$ is replaced by $\eps^{-1}f$. This rescaling does not change any physical quantity as our problem is linear, and leaves $J$ invariant. 

Because our original problem involves different orders of differentiation, the rescaling imposes a combination of terms with different powers of $\eps$.  Note, however, that the term $\eps^2 \aver{S_\eps Ju_\eps^T,S_\eps \tilde Ju_\eps^T}$ is not small since the coefficients $u_\eps$ are differentiated and $\eps\nabla u_\eps(x) = (\eps\nabla)[u(\frac x\eps)] = (\nabla u)(\frac x\eps)$. First-order derivatives in $J$ or second-order derivatives in $f$ are indeed of order $\eps$ and so do not contribute to the homogenized limit. The above problem is therefore degenerate and bears similarities with the example treated in \cite[Chapter 14]{BLP-78}. In the homogenized limit, only linear profiles for $\psi$ and $f$ remain, which results in an effective second-order elliptic equation for $f$ even though higher-order derivatives are present in the $\eps-$dependent problem.

\medskip

As we mentioned earlier, the homogenization theory of many problems with periodic (or more general stationary) coefficients is very well understood. We refer the reader to \cite{BLP-78,JKO-SV-94} and their extensive bibliographies for relevant material in what follows. Among all questions one can ask in homogenization theory, we focus on one, namely the convergence of the heterogeneous solution to its homogenized solution on the domain $X$ when Dirichlet boundary conditions are prescribed. Here, convergence will be shown for the variables $f_\eps$ strongly in the $L^2$ sense and weakly in the $H^1$ sense. To obtain this result, we use Tartar's energy method. The derivation closest to ours is a model studied in  \cite[Chapter 14]{BLP-78}, in which the Tartar method is analyzed to solve a scalar degenerate homogenization problem.

In this problem, as in many other homogenization settings, we first need to obtain a homogenized coefficient, not unlike $\bar\fa$ in earlier sections, which is written as an appropriate integral of an appropriate {\em corrector}. For a periodic domain $(0,1)^D$ and $c$ a constant tensor in $\Rm^{D\times m}$, let the corrector $J=J(c)$ be the unique solution in $\mH_\sharp$ of
\begin{equation}\label{eq:corper1}
  \fa(c+J,\tilde J) =0 ,\quad \mbox{ for all } \tilde J\in \mH_\sharp.
\end{equation}
Then $\bar \fa$ is the unique homogenized (matrix-valued) coefficient characterized by
\begin{equation}\label{eq:coefper}
  (c,\bar \fa c) = \fa(c+J,c+J) .
\end{equation}
The coefficient $\bar\fa$ is the same as $\bar\fa_\sharp$ constructed in the preceding section.

For $\eps>0$, we consider for concreteness the problem posed on an arbitrary open bounded domain $X$ with periodic coefficients and with boundary conditions restricting $f$ to be a smooth function $f_0$ on $\partial X$.  Let $J_0=\nabla\cdot f_0$ and $J_\eps=\nabla\cdot f_\eps$ be the unique solution in $J_0+\mH_D$ of the heterogeneous problem
\begin{equation}\label{eq:epsDir}
  \fa_\eps (J_\eps , \tilde J) =0 \quad \mbox{ for all } \tilde J \in\mH_D.
\end{equation}
This problem admits a unique solution with $f_\eps=f_0$ on $\partial X$ as an application of the above Riesz representation theorem. Then, adapting  the standard Tartar energy method \cite[Chapter 14]{BLP-78}, we obtain the following result:
\begin{theorem}\label{thm:perhom}
 Let $f_\eps$ be constructed as above. Then $f_\eps$ converges strongly as $\eps\to0$ in the $L^2(X)$ topology to a limit $f$ with $J=\nabla\cdot f$ the unique solution in $J_0+ \mH^1_D$ of
\begin{equation}\label{eq:limhom}
 \fa_0(J , \tilde J) := \aver{ \bar\fa J ,\tilde J} = 0 \quad \mbox{ for all } \tilde J  \in \mH_D(X).
\end{equation}
This may be recast as the system of partial differential equations $-\nabla \cdot\left(\bar \fa \nabla\cdot f\right)=0$ in $X$ with $f=f_0$ on $\partial X$. 
\end{theorem}
\begin{proof}
We apply the energy method of Tartar to degenerate systems of equations following the presentation in, e.g., \cite[Chapter 14]{BLP-78}, which treats a degenerate scalar equation.
Let $f_\eps$ be the unique solution obtained in \eqref{eq:epsDir} and equal to $f_0$ on $\partial X$.
By assumption, the coefficients are $1-$periodic and $\fO_\eps= \fO>0$ and thus, as an application of \eqref{eq:bdaeps} and the fact that $f$ is an anti-symmetric tensor, have that
\[
   \| \nabla f_\eps \|_{L^2} \leq C.
\]
Here, we use the bound obtained from $\fa_\eps$ as well as the fact that $\nabla\cdot f$ controls the $H^1$ norm of $f$ for anti-symmetric tensors in $\mH$.
By an application of the Poincar\'e inequality, this means that $f_\eps$ is bounded in the $H^1$ sense as well and we deduce the convergence of $f_\eps$ to $f$ weakly in $H^1$ and strongly in $L^2$ (for each component):
\[
  f_\eps \rightharpoonup f,\quad \mbox{ in  the $H^1$ sense},\qquad f_\eps\to f\quad \mbox{ in  the $L^2$ sense}.
\]
We also deduce from the equations and bounds on the coefficients that
\[
 \eps \| [\nabla \nabla\cdot f]u_\eps^T \|_{L^2} +   \eps \| \nabla [(\nabla\cdot f) u_\eps^T] \|_{L^2} \leq C.
\]
We write the equation for $J_\eps$ as
\[
  \aver{w_\eps J_\eps w_\eps^T + \eps^2 [S_\eps u_\eps^T]^T S_\eps [J_\eps u_\eps^T], \tilde J }
  + \aver{ \eps S_\eps [J_\eps u_\eps^T] ,  \eps S_\eps [\tilde J] u_\eps^T } =0.
\]
We use here square brackets to delimit which terms the first-order differentiations in $S_\eps$ act on. We recast the first term as $\aver{\xi^\eps,\tilde J}$ and with obvious notation as $\aver{\xi^\eps_w+\xi^\eps_u,\tilde J}$. Taking $\tilde J$ sufficiently smooth, the last contribution converges to $0$:
\[
  \aver{\xi^\eps,\tilde J} = o(1).
\]
The above bounds on $f_\eps$ show that $\xi^\eps$ (as well as $\xi^\eps_w$ and $\xi^\eps_u$) is bounded in the $L^2$ sense and hence converges weakly in that sense to $\xi=\xi_w+\xi_u$. Passing to the limit therefore gives
\begin{equation}\label{eq:conserv}
  \aver{\xi,\nabla\cdot \tilde f} =0,
\end{equation}
for all $\tilde f$ sufficiently smooth, and therefore by density for any $\tilde f\in\mH^1_D$, the closure of $\mH_D$ for the $H^1$ topology.

It therefore remains to show that $\xi$ equals $\bar \fa \nabla\cdot f$ for $f$ the limit of $f_\eps$.

Let us introduce the anti-symmetric linear tensor $\fc=\fc(x)$ such that $\nabla\cdot \fc(x)=c \in \Rm^{D\times m}$ for $x\in\Rm^D$ and for any prescribed tensor $c$.

We denote by $\theta=\theta(c)$ the {\em harmonic} coordinates such that $\nabla\cdot\theta\in c+\mH_\sharp$ given as the unique solution to
\begin{equation}\label{eq:corrector}
   \fa(\nabla\cdot\theta,\tilde J)=0,\quad\mbox{ for all } \tilde J \in \mH_\sharp.
\end{equation}
The solution $\theta-\fc(x)$ is referred to as the {\em corrector}, a periodic, and hence (locally in $L^2$) bounded, perturbation of the linear profile $\fc(x)$. Let us now define $\theta_\eps=\eps \theta(\frac\cdot\eps)$. Note that by linearity, $\eps\fc(\frac\cdot\eps)=\fc$. However, $\eps\theta(\frac\cdot\eps)-\fc(x)$ converges to $0$ in the $L^2$ sense (locally) since $\theta-\fc$ is periodic. Its spatial gradient is given by $(\nabla\theta)(\frac\cdot\eps)-c$, which is of order $O(1)$. It is this separation of scales that makes homogenization on a periodic domain reasonably tractable. We then observe that the rescaled corrector solves the equation
\[
 \fa_\eps(\nabla\cdot \theta_\eps,\tilde J)=0, \quad \mbox{ for all } \tilde J\in \mH_D.
\]

One last classical ingredient that makes calculations straightforward in the periodic setting is the fact that for any periodic locally square integrable function $\zeta(x)$, we find that $\zeta(\frac\cdot\eps)$ converges weakly in the $L^2$ sense to its average $\int_Q \zeta(y)dy$ as $\eps\to0$, with $Q=(0,1)^D$. This is often referred to as a (generalized) Riemann-Lebesgue lemma.

Let us finally introduce a smooth test function $\phi\in C^\infty_0(X)$ to localize constraints in $X$. We then consider the equation for $f_\eps$ with test function $\nabla\cdot(\phi \theta_\eps) \in \mH_D$ and the equation for $\theta_\eps$ with test function $\nabla\cdot(\phi f_\eps)\in \mH_D$. Therefore, as a sort of Green identity, we have
\begin{equation}
  \fa_\eps(\nabla\cdot f_\eps,\nabla\cdot (\phi \theta_\eps)) - \fa_\eps(\nabla\cdot \theta_\eps,\nabla\cdot (\phi f_\eps))L:= \delta\fa_\eps^w + \delta\fa_\eps^u=0. \label{eq:aepsilongreen}
\end{equation}
For later convenience, we defined $\delta\fa_\eps^w$ to denote the contributions of the $w$-terms in the above sum
and $\delta\fa_\eps^u$ that of the $u$-terms.  (\ref{eq:aepsilongreen}) holds for every $\phi\in C^\infty_0(X)$ and every tensor $c\in\Rm^{D\times m}$, which will be sufficient to characterize $\xi$, the limit of $\xi^\eps$. 

Let us first consider the contribution of the $w$-terms to (\ref{eq:aepsilongreen}):
\begin{align}
\delta\fa_\eps^w &= \aver{\nabla \cdot f_\eps w_\eps^T, \nabla \cdot (\phi \theta_\eps)  w_\eps^T} - \aver{\nabla \cdot \theta_\eps w_\eps^T, \nabla \cdot (\phi f_\eps)  w_\eps^T} \notag \\
  &= \aver{\nabla \cdot f_\eps w_\eps^T, (\nabla\phi)\cdot \theta_\eps  w_\eps^T} - \aver{\nabla \cdot \theta_\eps w_\eps^T, (\nabla \phi) \cdot f_\eps  w_\eps^T}. \label{eq:thm41wconv}
\end{align}
Since $\theta_\eps$ converges strongly to $\fc$ locally and $\xi^\eps_w$ converges weakly to $\xi_w$, we have
\begin{align*}
 \aver{\nabla \cdot f_\eps w_\eps^T, (\nabla\phi)\cdot \theta_\eps  w_\eps^T} &\rightarrow  \aver{\xi_w,(\nabla\phi) \fc}= -\aver{\xi_w,c\phi}-\aver{\nabla\cdot \xi_w,\phi \fc}.
\end{align*}
The contribution from the second term of (\ref{eq:thm41wconv})  is  of the form $f_\eps$, which converges strongly to $f$, times a periodic object, which converges weakly to its limit. Denote the spatial average $ \aver{w\nabla\cdot \theta w^T} :=\bar\fa_w c$. We obtain 
\[
\aver{\nabla \cdot \theta_\eps w_\eps^T, (\nabla \phi) \cdot f_\eps  w_\eps^T} \rightarrow \aver{\bar\fa_w,\nabla \phi \cdot  f c}=-\aver{\bar\fa_w \nabla \cdot f, c\phi}
\]
where in the last step we integrated by parts (recall $\bar\fa_w$ is constant).  The limit of the first contribution is therefore overall
\begin{equation}
\delta\fa_\eps^w =   -\aver{\xi_w,c\phi}-\aver{\nabla\cdot \xi_w,\phi \fc} + \aver{\bar\fa_w \nabla \cdot f, c\phi}. \label{eq:thm41wfinal}
\end{equation}

Next we evaluate $\delta\fa_\eps^u$.  The only difference is the presence of higher-order derivatives in the operator $S_\eps$, which cause the problem to be degenerate:
\begin{eqnarray}
 \delta\fa_\eps^u&=& \aver{ \eps S_\eps [\nabla \cdot f_\eps u_\eps^T], \eps S_\eps [\nabla \cdot (\phi \theta_\eps)  u_\eps^T]} - \aver{\eps S_\eps [\nabla \cdot \theta_\eps u_\eps^T], \eps S_\eps [\nabla \cdot (\phi f_\eps)  u_\eps^T]}
\label{eq:thm41uconv}\\
 &=& \aver{ \eps S_\eps [\nabla \cdot f_\eps u_\eps^T], \eps S_\eps [(\nabla \phi) \cdot  \theta_\eps  u_\eps^T]} - \aver{\eps S_\eps [\nabla \cdot \theta_\eps u_\eps^T], \eps S_\eps [(\nabla \phi) \cdot  f_\eps  u_\eps^T]} + o(1), \nonumber
\end{eqnarray}
since $\eps S_\eps \phi = O(\eps)$ is lower order. Similarly, $\eps S_\eps \nabla\phi =O(\eps)$  and $\eps S_\eps \theta_\eps =O(\eps)$.  Therefore, the first term of (\ref{eq:thm41uconv}) converges to: 
\begin{align*}
\aver{ \eps S_\eps [\nabla \cdot f_\eps u_\eps^T], \eps S_\eps [(\nabla \phi) \cdot  \theta_\eps  u_\eps^T]} &= \aver{ (\eps S_\eps u_\eps^T)^T \eps S_\eps [\nabla \cdot f_\eps u_\eps^T],  (\nabla  \phi)\cdot \theta_\eps) } +o(1)\\
 &=: \aver{\xi_u^\eps,\nabla \phi \cdot \theta_\eps} +o(1)\rightarrow -\aver{\xi_u,c\phi}-\aver{\nabla\cdot \xi_u,\phi \fc}.
\end{align*}
since, as previously, $\nabla\phi \cdot \theta_\eps=\nabla\phi \cdot \fc+O(\eps)$.
Next we bound the second term of (\ref{eq:thm41uconv}).  Since $\eps S_\eps f_\eps = O(\eps)$ in the $L^2$ sense, and $\eps S_\eps \phi = O(\eps)$ as before, 
\[
\aver{\eps S_\eps [\nabla \cdot \theta_\eps u_\eps^T], \eps S_\eps [(\nabla \phi) \cdot  f_\eps  u_\eps^T]} = \aver{(\eps S_\eps u_\eps^T)^T \eps S_\eps [\nabla\cdot \theta_\eps u_\eps^T] ,\nabla \phi \cdot f_\eps} + o(1).
\]
 Again, this is the product of $f_\eps$, which converges strongly to $f$, and a periodic object, which converges weakly to a constant tensor $\aver{(\eps S_\eps u_\eps^T)^T \eps S_\eps [\nabla\cdot \theta_\eps u_\eps^T]} := \bar\fa_u c$.  Hence,
 \[\aver{\eps S_\eps [\nabla \cdot \theta_\eps u_\eps^T], \eps S_\eps [(\nabla \phi) \cdot  f_\eps  u_\eps^T]} \rightarrow \aver{\bar\fa_u, \nabla\phi \cdot fc } =-\aver{\bar\fa_w \nabla \cdot f, c\phi} .
 \]
 Therefore \begin{equation}
     \delta\fa_\eps^u =   -\aver{\xi_u,c\phi}-\aver{\nabla\cdot \xi_u,\phi \fc} + \aver{\bar\fa_u \nabla \cdot f, c\phi}. \label{eq:thm41ufinal}
 \end{equation}

Combining (\ref{eq:aepsilongreen}), (\ref{eq:thm41wfinal}) and (\ref{eq:thm41ufinal}), and defining $\bar\fa=\bar\fa_u+\bar\fa_w$: 
\[
   -\aver{\xi,c\phi} + \aver{\xi , \nabla \cdot (\phi \fc)} + \aver{\bar\fa \nabla\cdot f,c\phi} =0.
\]
 Since the middle term vanishes thanks to \eqref{eq:conserv} and the above holds for each tensor $c$ and test function $\phi$, we deduce that $\xi=\bar\fa \nabla\cdot f$, which concludes our proof.
\end{proof}

In spite of the structure of the original problem, with more derivatives applied to $Ju^T$ than $Jw^T$, in the high frequency limit, the problem is best described as a degenerate homogenization problem with a standard second-order elliptic (non-degenerate) limit.

Moreover, we can come back to get an effective equation for the solution $\psi_\eps$ given by \eqref{eq:psi} and observe in the above proof that $\nabla\psi_\eps$ is given by $\xi^\eps$ up to negligible contributions. Therefore, it is given in the limit by $\xi=\bar\fa\nabla\cdot f$, or in other words, $\nabla\cdot f =\bar\fa^{-1}\nabla\psi$, so that taking divergence again (which in our notation vanishes as an application of an exterior derivative twice), we get
\[
  -\nabla\cdot \bar\fa^{-1}\nabla \psi =0
\]
on $X$. The heterogeneous problem for $(\psi_\eps,v_\eps)$ involves the necessary coupling with a mitigating velocity field $v$ that mixes all components of $\psi$. While the problem is clearly degenerate for $\psi$ alone as the coefficient $a$ is of rank $m-1$, we obtain in the limit $\eps\to0$ an effective equation solely involving the variables $\psi$.

The expressions $v_\eps$ and $\nabla \psi_\eps$ are oscillating, for instance $v_\eps = J_\eps u_\eps^T$ with $J_\eps=\nabla\cdot f_\eps$ an object of order $O(1)$ in the periodic setting converging weakly to its limit. Many similar results proved in other periodic homogenization contexts, such as convergence of gradients and order $O(\eps)$ rates of convergence away from boundaries, are certainly valid here. We refer the reader to the corresponding literature, e.g., \cite{BLP-78,JKO-SV-94}, for details.

Let us briefly make a final comment on the comparison of homogenized coefficients obtained with periodic and Dirichlet boundary conditions in the homogenization regime. We saw in Proposition \ref{prop:comp} that $\bar\fa_D\geq \bar \fa_\sharp$. Let $X$ be the unit cube $(0,1)^D$ with $\eps$-oscillatory coefficients. The periodic coefficient $\bar\fa_\sharp$ is given in \eqref{eq:coefper} with $f$ the solution of \eqref{eq:corper1}. For $\eps=\frac1N$, and since all coefficients are periodic, we find that 
\[
  (c,\bar\fa_\sharp c) = \fa_\eps\left (c+\eps\nabla\cdot f\left(\frac x\eps\right),c+\eps\nabla\cdot f\left(\frac x\eps\right)\right).
\]
Let us define $f_\eps=\eps f(\frac\cdot\eps)$ and introduce $0\leq \theta_\eps(x)\leq 1$ a smooth function equal to $1$ $\eps-$away from $\partial X$ and equal to $0$ in the vicinity of $\partial X$. Then $\nabla\cdot(\theta_\eps f_\eps)\in \mH_D$ and $\nabla \cdot (\theta_\eps f_\eps) = \theta_\eps\nabla\cdot f_\eps + \nabla \theta_\eps \eps f(\frac\cdot \eps)$. The latter is controlled in the $\eps-$vicinity of $\partial X$. Therefore, going through similar details to those in the above proof, we verify that 
\[
  (c,\bar\fa_\sharp c)  \geq  \fa_\eps\big (c+\nabla\cdot \theta_\eps f_\eps,c+\nabla\cdot\theta_\eps f_\eps\big) + o(1).
\]
But the latter is smaller than the coefficient obtained with Dirichlet boundary conditions, which we will call $\bar\fa_{ND}$.

Let now $\bar \fa_{N\sharp}$ be the homogenized coefficient obtained with periodic boundary conditions on $[0,N]^D$. This is obtained as the solution to the problem $\fa_N(c+\nabla\cdot f, \tilde J)=0$ for all $\tilde J\in\mH_\sharp$. The solution $f_1$ obtained by periodicity on the cube $(0,1)^D$ and extended by periodicity solves the partial differential equation
\[
    c+J_1 w^Tw + L (J_1u^T)u=0
\]
with $J_1=\nabla\cdot f_1$, both on the unit cube as well as on the large cube $[0,N]^D$; for this, we need the solution to be sufficiently smooth, which we assume. By uniqueness of the solution, we deduce that 
\[
  (c,\bar \fa_{N\sharp}c)  =  N^{-D} \fa_N(c+J_1,J_1) = (c,\bar \fa_{\sharp}c).
\]
This shows that $\fa_\sharp=\fa_{N\sharp}\leq \fa_{D\sharp} = \fa_\sharp+o(1)$ and hence the equality of all coefficients in the limit $\eps\to0$. Note that the regularity of the solution is essential to obtain that $\fa_\sharp=\fa_{N\sharp}$. For functionals that are not convex and hence with less clear uniqueness properties, the result may simply be false \cite{muller1987homogenization}.

\section{Remarks on homogenization in stationary media}
\label{sec:rand}

We now briefly consider the extension of homogenization results in the setting of stationary ergodic underlying coefficients. A number of methods and results have been generalized from the periodic to the random settings. The energy method of Tartar, which we used in the preceding section, was generalized to the random setting in \cite{Koz-MSb-79,PV-RF-81}; see also the monograph \cite{JKO-SV-94}. A general method of stochastic two-scale convergence was proposed in \cite{BMW-JRAM-94}. In the meantime, a method based on the minimization framework we considered so far in this paper was developed in \cite{dal1986nonlinear,dal1986nonlinear1}; see also the monograph \cite{dal2012introduction}. 

The method proposed in the latter references is based on the notion of $\Gamma-$convergence \cite{dal2012introduction} of functionals and the property that under some general settings, such functionals $\Gamma-$converge as soon as their minimal values converge. The $\Gamma$-convergence of the functionals in turn implies that of minimizing sequences and yields results similar to the ones obtained in the preceding section by Tartar's energy method. That the minimal values converge, in the translation invariant ergodic setting, is a consequence of a general sub-additive property first used in the homogenization context in \cite{dal1986nonlinear}. It is this subadditive property that we wish to focus on here. We show below that effective coefficients obtained from variational problems with Dirichlet conditions do satisfy the appropriate sub-additivity properties. We do not consider the extension of $\Gamma-$convergence results to our degenerate setting, and instead refer the reader to \cite{dal2012introduction} in the setting of functionals that depend on $\nabla\cdot f$ only; not on terms of the form $S(\nabla \cdot f)u^T$.

The above minimization method has also proved fruitful when one is interested in corrections beyond the homogenization limit. While the latter implies that a solution $f_\eps$ converges to a limit $f_0$ in some (strong $L^2$, say) sense, the corrections aim to understand the size of the error $f_\eps-f_0$ and, if possible, to characterize the limiting law of this appropriately rescaled error term. This is a quite difficult problem, with recent progress obtained using a minimization approach \cite{armstrong2017quantitative}; see also \cite{GO-AAP-12,GM-MMS-16} for a non-degenerate homogenization problem with short-range random coefficients. The structure of the random fluctuations is much more complex than that in the periodic setting, where expansions show an error term of order $O(\eps)$, at least away from boundaries where boundary layer terms need to be included \cite{BLP-78,JKO-SV-94}. Random fluctuations in general stationary ergodic settings can be as large (in powers of $\eps$) as one wishes, as shown in some concrete examples in, e.g., \cite{B-CLH-08,BGGJ-AA-12,BGMP-AA-08}.

Although several of the aforementioned results most likely apply to our setting of Dirac fluids, we do not consider them here and focus on establishing the sub-additivity property that is a necessary step toward a $\Gamma-$convergence homogenization result. It also shows how coefficients, for instance computed numerically, are expected to behave as the domain size tends to infinity and the realization of the random medium varies. 

\medskip

Let us briefly recall a standard construction of stationary coefficients \cite{armstrong2017quantitative}.  We assume a probability space $(\Omega,\mF,\Pm)$ 
large enough to admit a group of discrete $D-$dimensional shifts $\tau_{n}$ with $n\in\Zm^D$ that are measure-preserving, i.e., $\Pm=\Pm\circ \tau_n^{-1}$ for each $n\in\Zm^D$, and with the abelian group structure $\tau_{n+m}=\tau_n\tau_m$. Intuitively, $\omega\in\Omega$ describes a particular realization of random functions $a$, $b$, etc., defined on $\mathbb{R}^D$; $\tau_n a = a(x+n)$ denotes a discrete translation. More precisely, let
 $z(x,\omega)$ be a coefficient that is for each $x\in\Rm^D$ a (real-valued) random variable defined on $(\Omega,\mF,\Pm)$ and for each $\omega$ a (Borel) measurable function that is continuous, uniformly bounded in $x$ and for some coefficients (e.g., $a,b$ and hence $u$) continuously differentiable. The coefficient $z$ is stationary when its joint distribution for any finite number of values of $x_j$ is identical to the joint distribution with $x_j$ replaced by $\tau_{-n}x_j$ for any $n\in\Zm^D$ . We assume all the coefficients $(\eta,\zeta,a,b)$, and hence $(u,w)$, to be stationary. We refer to the above mentioned literature for many examples of construction of spaces on which random coefficients are naturally stationary. Once the coefficients $z(x,\omega)$ are defined in $\Omega\times\Rm^D$, we define as in the preceding section the rescaled coefficients $z_\eps(x,\omega)=z(\frac x\eps,\omega)$.

The final notion of importance here is that of ergodicity, which means that for each $A\subset\mF$ that is invariant under $\tau$ (in the sense that $\tau_n^{-1}A=A$ for each $n\in\Zm^d$), then $\Pm(A)=0$ or $\Pm(A)=1$.  Many constructions in the above mentioned references are proved to be ergodic as well as stationary. Ergodicity precludes the existence of subsets $A\in\mF$ in $\Omega$ of intermediate measure between $0$ and $1$ such that $A$ and its complementary $A^c=\Omega-A$ are both invariant for the transformations $\tau$, whereby separating averaging into two non-communicating parts of the state space and therefore leading to two potentially different homogenization limits. 

\medskip

Armed with these constructions, we consider on $X=(0,1)^D$ the unit torus and for each realization $\omega\in\Omega$ the sequence of problems:
\[
  \mbox{ Find } J_\eps\in\mH_D\ \mbox{ such that } \ \fa_\eps (c+ J_\eps, \tilde J) =0 \ \mbox{ for all } \tilde J\in \mH_D.
\]
Since $\fa_\eps$ is a bilinear symmetric form, the above problem may be recast as the unique solution to the minimization of
\[
   \min_{J\in\mH_D} \frac12 \fa_\eps(J,J) - \fa_\eps(c,J).
\]
The homogenized coefficient at a given scale $\eps$ is given by
\[
    (c, \bar \fa_\eps c)  = \fa_\eps(c+J_\eps , c+J_\eps)
\]
for $J_\eps$ the above solution. Combining the two expressions as in Lemma \ref{lem:minim}, we observe that
\[
   (c, \bar \fa_\eps c) = \min_{J-c \in \mH_D}  \fa_\eps(J,J).
\]
In the periodic setting, we saw that $\bar \fa_\eps$ was equal to the homogenized coefficient $\bar\fa$ by construction. In the random setting, where the coefficient are no longer periodic, the above quantity depends on $\eps$ and on the realization of the random medium. However, as first exploited in \cite{dal1986nonlinear} in the homogenization context, the above problem with Dirichlet boundary conditions admits very favorable subadditivity properties. Such properties are not directly shared by a problem written with periodic boundary conditions. This was the main motivation for us to develop the setting with Dirichlet boundary conditions.

The random coefficients constructed on $(\Omega,\mF,\Pm)$ are assumed to be invariant by discrete translations, as described above. In addition, we assume that $\Pm-$almost surely, the coefficients satisfy the local stability condition \eqref{eq:osc} with $\fO>0$ uniformly on each cube of the domain before rescaling. As in the periodic setting, this allows one to control $\aver{|J|^2}$ by $a_\eps(J)$ almost surely.  
The construction of such coefficients $b(x,\omega)$ may be based on those given on checkerboards. Instead of constructing a constant coefficient that is piecewise constant on the piece $(0,1)^D$, we construct a smooth coefficient with support $(-1,2)^D$ that is sufficiently oscillatory in all directions. Coefficients $b$ and $a$ so constructed then satisfy the bound \eqref{eq:osc}.


\begin{theorem}
  Let us assume that the coefficients defined on $(\Omega,\mF,\Pm)$ are stationary and that $\fa_\eps$ is uniformly an inner product on $\mH_D$. Then $\Pm-$ almost surely,
\begin{equation}\label{eq:stochlimit}
  \lim_{\eps\to0}  \min_{J-c \in \mH_D}  \fa_\eps(J,J) =  (c, \bar \fa(\omega) c ),
\end{equation}
for a tensor $\bar\fa(\omega)$. Moreover, if the translations on $(\Omega,\mF,\Pm)$ are ergodic, then the limiting tensor $\bar \fa$ is independent of $\omega$ almost surely.
\end{theorem}

\begin{proof}
 This is a direct application of the subadditive ergodic theory \cite{dal1986nonlinear}. As we already noted, the left-hand side, for each $\eps$, is a quadratic form in $c$. We now prove that the quantity is subadditive. To do so, we define $N=\eps^{-1}$ and consider the domain $X_N=(0,N)^D$. Let $\fa[X]$ be the variational form (written with $\eps=1$) written. Then
\[
 \min_{J-c \in \mH_D} \frac12 \fa_\eps(J,J) = \min_{J-c \in \mH_D[X_N]}\frac{1}{2N^D} \fa[X_N](J,J) ,
\]
where $\mH_D[X]$ is the Hilbert space constructed on $X$. Let us define
\[
   \nu[X] =  \min_{J-c \in \mH_D[X]}\frac{1}{2} \fa[X](J,J).
\]

Let $X$ be a bounded open subset in $\Rm^D$, a union of disjoint open subsets $X_j$ in the sense that $\lambda(X - \bigcup_j X_j)=0$, for $\lambda$ the $D-$dimensional Lebesgue measure. We then verify that for any function $\nabla\cdot f_j\in \mH_D[X_j]$, hence with $f_j$ vanishing on $\partial X_j$ as well as $(\nabla\cdot f)u^T$, then the function $f$ equal to $f_j$ on each $X_j$ and equal to $0$ on $X - \cup_j X_j$ is such that $J=\nabla\cdot f$ belongs to $\mH_D[X]$.  This shows that
\[
 \nu[X] \leq \dsum_j \nu[X_j],
\]
since the minimization over $X$ involves a larger class of functions than those obtained from each subset. We then apply the sub-additive ergodic theorem as in \cite[Proposition 1]{dal1986nonlinear} following \cite{akcoglu1981ergodic} to obtain that for $X$ an open subset,
\[
  \dfrac{1}{|tX|} \nu[tX](\omega) \to \nu(\omega)
\]
as $t\to\infty$ for $\omega\in\Omega'$ a set of full measure. Moreover, when the coefficients are ergodic, then the above quantity $\nu$ is independent of $\omega$ a.s. This shows \eqref{eq:stochlimit}.
\end{proof}

Note that the above sub-additive property does not (necessarily) hold for periodic boundary conditions. Indeed, the juxtaposition of periodic solutions may jump across interfaces and therefore does not necessarily form a valid test function on the larger domain. It is quite likely that the calculation of homogenized coefficients may also be performed using periodic boundary conditions as in the periodic setting. In the limit $\eps\to0$, we expect all procedures to converge to the same value as in the periodic setting.  We do not consider such results further and refer to the appropriate literature \cite{dal2012introduction,muller1987homogenization} for pointers to the required mathematical tools.


\addcontentsline{toc}{section}{References}
\begin{thebibliography}{99}

\bibitem{akcoglu1981ergodic}
{\sc M.~A. Akcoglu and U.~Krengel}, {\em Ergodic theorems for superadditive
  processes}, J. reine angew. Math, 323 (1981), pp.~106--127.

     \bibitem{andreev}
  {\sc A.~V.~Andreev, S.~A.~Kivelson and B. Spivak}, {\em Hydrodynamic description of transport in strongly correlated electron systems}, Phys. Rev. Lett., 106 (2011), p.~256804.
  
  \bibitem{armstrong2017quantitative}
{\sc S.~Armstrong, T.~Kuusi, and J.-C. Mourrat}, {\em Quantitative stochastic
  homogenization and large-scale regularity}, Springer-Verlag, New York, 2019.
  
  \bibitem{ashcroft}
{\sc N.~W.~Ashcroft and N.~D.~Mermin}, {\em Solid-State Physics}, Brooks Cole, 1976.



\bibitem{B-CLH-08}
{\sc G.~Bal}, {\em Central limits and homogenization in random media},
  Multiscale Model. Simul., 7(2) (2008), pp.~677--702.

\bibitem{BGGJ-AA-12}
{\sc G.~Bal, J.~Garnier, Y.~Gu, and W.~Jing}, {\em Corrector theory for
  elliptic equations with long-range correlated random potential}, Asymptot.
  Anal., 77 (2012), pp.~123--145.

\bibitem{BGMP-AA-08}
{\sc G.~Bal, J.~Garnier, S.~Motsch, and V.~Perrier}, {\em Random integrals and
  correctors in homogenization}, Asymptot. Anal., 59(1-2) (2008), pp.~1--26.

\bibitem{BLP-78}
{\sc A.~Bensoussan, J.-L. Lions, and G.~C. Papanicolaou}, {\em Asymptotic
  analysis for periodic structures}, in Studies in Mathematics and its
  Applications, 5, North-Holland Publishing Co., Amsterdam-New York, 1978.

\bibitem{BMW-JRAM-94}
{\sc A.~Bourgeat, A.~Mikeli\'c, and S.~Wright}, {\em Stochastic two-scale
  convergence in the mean and applications}, J. reine angew. Math, 456 (1994),
  pp.~19--51.

\bibitem{dal2012introduction}
{\sc G.~Dal~Maso}, {\em An introduction to $\Gamma$-convergence}, vol.~8,
  Springer Science \& Business Media, 2012.

\bibitem{dal1986nonlinear1}
{\sc G.~Dal~Maso and L.~Modica}, {\em Nonlinear stochastic homogenization},
  Annali di matematica pura ed applicata, 144 (1986), pp.~347--389.

\bibitem{dal1986nonlinear}
\leavevmode\vrule height 2pt depth -1.6pt width 23pt, {\em Nonlinear stochastic
  homogenization and ergodic theory.}, Journal f{\"u}r die reine und angewandte
  Mathematik, 368 (1986), pp.~28--42.

  
        \bibitem{fritz}
  {\sc L.~Fritz, J.~Schmalian, M.~M\"uller and S.~Sachdev}, {\em Quantum critical transport in clean graphene}, Phys. Rev., B78 (2008), p.~085416.
  
\bibitem{GO-AAP-12}
{\sc A.~Gloria and F.~Otto}, {\em An optimal error estimate in stochastic
  homogenization of discrete elliptic equations}, Ann. Applied Probab., 22
  (2012), pp.~1--28.

\bibitem{GM-MMS-16}
{\sc Y.~Gu and J.-C. Mourrat}, {\em Scaling limit of fluctuations in stochastic
  homogenization}, Multiscale Modeling \& Simulation, 14  (2016), pp.~452--481.

      \bibitem{gurzhi}
  {\sc R.~N.~Gurzhi}, {\em Minimum of resistance in impurity-free conductors}, Soviet JETP, 17 (1963), p.~521.
  
  
          \bibitem{hartnoll2007}
  {\sc S.~A.~Hartnoll, P.~K.~Kovtun, M.~M\"uller and S.~Sachdev}, {\em Theory of the Nernst effect near quantum phase transitions in condensed matter, and in dyonic black holes}, Phys. Rev., B76 (2007), p.~144502.
  

    \bibitem{JKO-SV-94}
{\sc V.~V. Jikov, S.~M. Kozlov, and O.~A. Oleinik}, {\em Homogenization of
  differential operators and integral functionals}, Springer-Verlag, New York,
  1994.
  
  \bibitem{Koz-MSb-79}
{\sc S.~M. Kozlov}, {\em The averaging of random operators}, Math. Sb. (N.S.),
  109 (1979), pp.~188--202.

    \bibitem{landau}
{\sc L.~D.~Landau and E.~M.~Lifshitz}, {\em Fluid Mechanics}, Butterworth Heinemann, 2nd ed., 1987.

  
     \bibitem{lucas2015}
  {\sc A.~Lucas}, {\em Hydrodynamic transport in strongly coupled disordered quantum field theories}, New J. Phys., 17 (2015), p.~113007.
  
    \bibitem{lucas2016}
  {\sc A.~Lucas, J.~Crossno, K.~C.~Fong, P.~Kim and S.~Sachdev}, {\em Transport in inhomogeneous quantum critical fluids and in the Dirac fluid in graphene}, Phys. Rev., B93 (2016), p.~075426.
  
  \bibitem{lucasreview}
  {\sc A.~Lucas and K.~C.~Fong}, {\em Hydrodynamics of electrons in graphene}, J. Phys: Cond. Mat, 30 (2018), p.~053001.
  
    \bibitem{lucashartnoll2017}
  {\sc A.~Lucas and S.~A.~Hartnoll}, {\em Resistivity bound for hydrodynamic bad metals}, Proc. Nat. Acad. Sci., 114 (2017), pp.~11344--11347.
  
      \bibitem{lucashartnoll2018}
  {\sc A.~Lucas and S.~A.~Hartnoll}, {\em Kinetic theory of transport for inhomogeneous electron fluids}, Phys. Rev., B97 (2018), p.~045105.

\bibitem{muller1987homogenization}
{\sc S.~M{\"u}ller}, {\em Homogenization of nonconvex integral functionals and
  cellular elastic materials}, Archive for Rational Mechanics and Analysis, 99
  (1987), pp.~189--212.

\bibitem{PV-RF-81}
{\sc G.~C. Papanicolaou and S.~R.~S. Varadhan}, {\em Boundary value problems
  with rapidly oscillating random coefficients}, in Random fields, Vol. I, II
  (Esztergom, 1979), Colloq. Math. Soc. J{\'a}nos Bolyai, 27, North Holland,
  Amsterdam, New York, 1981, pp.~835--873.

\bibitem{Taylor-PDE-1}
{\sc M.~E. Taylor}, {\em Partial Differential Equations I}, Springer Verlag,
  New York, 1997.

\end{thebibliography}

%
%
%
%
%

%
%
%
\end{document}